\documentclass[reqno]{amsart}

\usepackage{mathrsfs, color}
\usepackage{amsfonts}
\usepackage{amssymb,amsmath}
\usepackage{amsthm}
\usepackage{graphics}
\usepackage{graphicx}
\usepackage{epsfig}
\DeclareMathOperator*{\cheng}{\text{\LARGE $\times$}}
\DeclareMathOperator*{\supp}{\text{supp}}
\DeclareMathOperator*{\argmax}{\text{argmax}}
\DeclareMathOperator*{\argmin}{\text{argmin}}

\newtheorem{theorem}{Theorem}[section]
\newtheorem{lemma}[theorem]{Lemma}
\newtheorem{corollary}[theorem]{Corollary}

\theoremstyle{definition}
\newtheorem{definition}[theorem]{Definition}

\theoremstyle{remark}
\newtheorem{remark}[theorem]{Remark}

\numberwithin{equation}{section}

\begin{document}

\title[Bayesian approach for inverse scattering problem]
 {Infinite-dimensional Bayesian approach for inverse scattering problems of a fractional Helmholtz equation}

\author[J.X.Jia]{Junxiong Jia}
\address{School of Mathematics and Statistics,
Xi'an Jiaotong University,
 Xi'an
710049, China; BCMIIS}
\email{jjx323@xjtu.edu.cn}
\thanks{}

\author[S. Yue]{Shigang Yue}
\address{Lincoln School of Computer Science,
University of Lincoln,
Brayford Pool, Lincoln LN6 7TS,
United Kingdom }
\email{syue@lincoln.ac.uk}

\author[J. Peng]{Jigen Peng}
\address{School of Mathematics and Statistics,
Xi'an Jiaotong University,
 Xi'an
710049, China}
\email{jgpeng@xjtu.edu.cn}

\author[J. Gao]{Jinghuai Gao}
\address{School of Electronic and Information Engineering,
Xi'an Jiaotong University,
 Xi'an
710049, China}
\email{jhgao@xjtu.edu.cn}


\date{}


\begin{abstract}
In this paper, we focus on a new wave equation described wave propagation in the attenuation medium.
In the first part of this paper, based on the time-domain space fractional wave equation, we formulate the frequency-domain equation named as fractional Helmholtz equation.
According to the physical interpretations, this new model could be divided into two separate models: loss-dominated model and dispersion-dominated model. For the loss-dominated model (it is an integer- and fractional-order mixed elliptic equation), a well-posedness theory has been established and the Lipschitz continuity of the scattering field with respect to the scatterer has also been established.
Because the complexity of the dispersion-dominated model (it is an integer- and fractional-order mixed elliptic system),
we only provide a well-posedness result for sufficiently small wavenumber.
In the second part of this paper, we generalize the Bayesian inverse theory in infinite-dimension to allow a part of the noise depends on the target function (the function needs to be estimated). Then, we prove that the estimated function tends to be
the true function if both the model reduction error and the white noise vanish. 
At last, our theory has been applied to the loss-dominated model with absorbing boundary condition.
\end{abstract}

\maketitle

\section{Introduction}\label{introductionSection}

Attenuation effect is an important phenomenon when we consider the wave propagation in some attenuation medium.
Numerous physical models have been proposed \cite{dixue5,dixue3,dixue4} in previous studies.
If we write the Helmholtz equation on inhomogeneous medium as follows:
\begin{align}\label{classical}
\Delta u + k^{2} n(x) u = 0,
\end{align}
then we may consider
$n(x) = n_{1}(x) + i\frac{n_{2}(x)}{k}$
to incorporate the case of the absorbing medium \cite{InverseScatter1}.
Hence, the attenuation problem may be incorporated into the classical studies on Helmholtz equations (e.g., \cite{Bao1,Bao2,Yang1}).
However, the attenuation effect indeed incorporate two effects: amplitude loss and velocity dispersion.
The aforementioned model (\ref{classical}) seems to mix these two effects together, 
so we can not study these two effects separately.

Space fractional wave equations, which can separate the two effects incorporated in the attenuation effect, have been proposed.
Before revealing the form of this new fractional model, we would like to provide an introduction of the fractional time wave equation.
Based on the Caputo's fractional derivative \cite{Podlubny},
the isotropic stress-strain ($\sigma$-$\epsilon$) relation could be deduced in the following form \cite{Caputo1}:
\begin{align*}
\sigma = \frac{M_{0}}{t_{0}^{-2\gamma}}\frac{\partial^{2\gamma}\epsilon}{\partial t^{2\gamma}}.
\end{align*}
Then the following Caputo's wave equation has been established
\begin{align}\label{caputoequ}
\begin{split}
\frac{\partial^{2-2\gamma}}{\partial t^{2-2\gamma}}u = c^{2}\omega^{-2\gamma}\Delta u,
\end{split}
\end{align}
where $c^{2}(x) = c_{0}^{2}(x)\cos^{2}(\pi \gamma /2)$. Here $c_{0}$ is the sound velocity.
$\gamma$ is a function related to the quality factor $Q$, that is $\gamma \rightarrow 0$ as $Q\rightarrow \infty$
and $\gamma \rightarrow 1/2$ as $Q\rightarrow 0$.
Then Carcione et al \cite{Carcione1,Carcione2} successfully solved the fractional time wave equation
using the Grünwald-Letnikow and central-difference approximations for the time discretization and Fourier
method to compute the spatial derivative.
The time fractional wave equation describes the constant $Q$ attenuation ($Q$ is constant in the frequency domain) precisely; however,
it is hard to solve and it also mixes amplitude loss effect and velocity dispersion effect together.

Based on the Caputo's wave equation (\ref{caputoequ}), after some intricate calculations in the angular and space frequency domain,
Zhu, Carcione, and Harris \cite{Zhu1,Zhu2} proposed the following space fractional model:
\begin{align}\label{fractionalFull}
\begin{split}
\frac{1}{c(x)^{2}}\frac{\partial^{2}}{\partial t^{2}}u = -\eta(x)(-\Delta)^{\gamma(x) + 1}u - \tau(x)\frac{\partial}{\partial t}(-\Delta)^{\gamma(x)+1/2}u,
\end{split}
\end{align}
with coefficients varying in space as follows:
\begin{align}\label{1coeff}
\begin{split}
\eta(x) = c_{0}(x)^{2\gamma(x)}\omega_{0}^{-2\gamma(x)}\cos (\pi\gamma(x)), \quad
\tau(x) = c_{0}(x)^{2\gamma(x)}\omega_{0}^{-2\gamma(x)}\sin(\pi\gamma(x)).
\end{split}
\end{align}
Here let us provide some explanations for the notations used in (\ref{fractionalFull}) and (\ref{1coeff}).
$\omega_{0}$ denotes a reference frequency, $c_{0}(x)$ denotes the phase velocity, and
$c(x)$ represents the space acoustic velocity.
The fractional power $\gamma(x)$ relates to the quality factor, as follows:
\begin{align}\label{1power}
\gamma(x) = \frac{1}{\pi}\arctan\left(\frac{1}{Q(x)}\right).
\end{align}
Obviously, we have $0<\gamma(x)<\frac{1}{2}$.
This model can be simplified as two separate models, namely, dispersion-dominated wave equation and
amplitude loss-dominated wave equation.
The dispersion-dominated wave equation has the following form:
\begin{align}\label{1dispersion}
\frac{1}{c^{2}}\frac{\partial^{2}}{\partial t^{2}}u = -\eta (-\Delta)^{\gamma + 1}u.
\end{align}
The amplitude loss-dominated wave equation has the following form
\begin{align}\label{1amplitude}
\frac{1}{c^{2}}\frac{\partial^{2}}{\partial t^{2}}u = \Delta u - \tau \frac{\partial}{\partial t}(-\Delta)^{\gamma + 1/2}u.
\end{align}
Hence, this space fractional model clearly has two advantages:
first, it can be solved quickly by spectral methods \cite{Zhu3} or other numerical methods;
second, it separates the dispersion effect and the amplitude loss effect, so researchers are able to analyze these two parts separately and
obtain a complete understanding of the attenuation effect.

Then, let us consider the time-harmonic solution of equation (\ref{fractionalFull}).
As usual, assuming the solution has the form
$e^{-i\omega t}u(x),$
then we derive an equation that could be called the fractional Helmholtz equation, as follows:
\begin{align}\label{1fullhelmholtz}
-\eta (-\Delta)^{\gamma + 1}u + i\omega\tau(-\Delta)^{\gamma+1/2}u + k^{2}(1+q(x))u = 0,
\end{align}
where $\omega$ denotes the angular frequency, $k$ represents the wavenumber, and $q(\cdot)$ is a function assumed to be larger than $-1$.
Equation (\ref{1fullhelmholtz}) could also be separated into two models: the loss-dominated model and the dispersion-dominated model.
More specifically, the loss-dominated fractional Helmholtz equation can be derived from equation (\ref{1amplitude}) as follow:
\begin{align}\label{1.1model}
\Delta u + i\omega\tau (-\Delta)^{\gamma+1/2}u + k^{2}(1+q(x))u = 0.
\end{align}
The dispersion-dominated fractional Helmholtz equation can be derived from equation (\ref{1dispersion}) as follows:
\begin{align}\label{1.1helmholtzEquation}
(-\Delta)^{\gamma+1}u - k^{2}(1+q)u = 0.
\end{align}

In this paper, under suitable assumptions stated in Assumption 1 in Section \ref{2forwardsection},
a well-posedness theory with general wavenumber $k>0$ has been constructed for equation (\ref{1.1model}).
For equation (\ref{1.1helmholtzEquation}), the problem seems to be difficult.
We can only constuct a unique solution for sufficiently small $k>0$.
Because the studies about fractional Laplace operator (a representation of non-local operators)
are a rather new topic in the field of elliptic partial differential equations,
the theories of this operator are little compared with the traditional second order elliptic operator.
Considering the difficulties brought by the fractional Laplace operator, our new results are non-trivial generalizations
of the results about second-order Helmholtz equation (\ref{classical}).

At this stage, the forward models proposed in this paper are clear, however, we are not content with this.
In the second part of this paper, we attempt to construct a Bayesian inverse theory for an inverse scattering problem
related to the fractional Helmholtz equation.

Now let us recall some basic developments in the Bayesian inverse theory.
Generally speaking, there are typically two philosophies in Bayesian inverse theory.
One philosophy involves discretizing the forward problem and then using Bayesian methodology to a finite-dimensional problem
(``discrete first, inverse second''[DFIS]).
Kaipio and Somersalo \cite{DisBayesian1} provide an excellent introduction for the DFIS method, especially large
inverse problems arising in differential equations.
The other philosophy involves constructing Bayesian inverse theory
in infinite-dimensional space, in which discretization of the continuous problem is postponed
to the final step (``inverse first, discrete second''[IFDS]).
The IFDS method could be dating back to 1970,
Franklin \cite{BayesianEarly} formulated PDE's inverse problems in terms of Bayes' formula on some Hilbert space.
Recently, Lasanen \cite{lasanen1,lasanen2,lasanen3,lasanen4} developed fully nonlinear theory.
Cotter, Dashti, Robinson, Stuart, Law, and Voss \cite{Bayesian2,MAP_detail,Bayesian3}
established a mathematical framework for a range of inverse problems for functions, given noisy observations.
They revealed the relationship between regularization techniques and the Bayesian framework.
In addition, the error of the finite-dimensional approximate solutions has been estimated.

In this study, we employ the IFDS method and construct the Bayesian theory of the inverse scattering problem.
Let $X,Y$ be separable Hilbert space, equipped with the Borel $\sigma$-algebra,
and $\mathcal{G}:X\rightarrow Y$ be a measurable mapping. Then, the inverse problem can be sought of as finding $x$ from $y$ where
\begin{align}\label{3.1jiaformula}
y = \mathcal{G}(x) + \eta,
\end{align}
and $\eta \in Y$ denotes noise.
An important assumption in the literature \cite{Bayesian2,MAP_detail,Bayesian3} is that the noise $\eta$ is independent of $x$.
However, in previous studies on inverse scattering problems, some model reduction errors may be brought into the forward problem
(e.g., the absorbing boundary condition has been employed in \cite{Bao2}).
By denoting the model reduction error as $\epsilon$,
we could reformulate equation (\ref{3.1jiaformula}) as follows:
\begin{align}\label{3.2jiaapproModel}
y = \mathcal{G}_{a}(x) + \epsilon + \eta
\end{align}
with $\mathcal{G}_{a}:X\rightarrow Y$ being a measurable mapping.
The error $\epsilon$ usually depends on $x$; thus, we need to generalize Bayesian inverse theory in infinite-dimensional space
to incorporate this situation.

From the principles of DFIS, a Bayesian approximation error approach is developed \cite{DisBayesian1,Bayesian8,Bayesian9},
which can be used to handle model approximate errors (not independent with the above mentioned variable $x$)
produced by some finite-dimensional approximations.
Acceptable inversion results can be obtained by this method with only a rough approximate forward solver,
so that it seems to be a promising method for inverse scattering problems.
However, there seems no special infinite-dimensional Bayesian inverse theory
for the model reduction error induced by the hypothesis of constructing the mathematical models, e.g.,
the error induced by some absorbing boundary conditions.

On the basis of the aforementioned considerations and the requirements for analyzing inverse scattering problems,
we modify the theory presented in \cite{Bayesian2,MAP_detail,Bayesian1,Bayesian3} to allow a part of the noise to depend on the state variable $x$.
Then, we prove that the estimated function tends to be the true function when both the model reduction error $\epsilon$ and the white noise $\eta$ vanish under a simple setting. Finally, we apply the theory to an inverse scattering problem related to equation (\ref{1.1model}).
In summary, the contributions of our work are as follows:
\begin{itemize}
  \item The well-posedness is obtained for a scattering problem related to the loss-dominated fractional Helmholtz equation.
  Based on the well-posedness result, the Lipschitz continuity of the forward map is obtained,
  which is useful for analyzing inverse scattering problems.
  \item A generalized infinite-dimensional Bayesian inverse method is developed,
  which can be called infinite-dimensional Bayesian model error method.
  In addition, its relationship with regularization methods is discussed.
  If both the model reduction error and the white noise vanish,
  it is proved that the estimated function tends to be the true function.
\end{itemize}

The contents of this paper are organized as follows.
In Section \ref{p2Pre}, notations are introduced and some basic knowledge of the fractional Laplace operator is presented.
In Section \ref{2forwardsection}, we construct the
well-posedness theory for the scattering field equation related to the loss-dominated equation firstly.
Secondly, we construct the well-posedness theory for a scattering field equation related to the dispersion-dominated equation with
sufficiently small wavenumber.
In Section \ref{3inverse}, we first derive the well-posedenss of the posterior measure when some model reduction errors are considered.
Then, we prove that the estimated solution tends to be the true function if both the model error and the white noise vanish.
In the last part of this section, the general theory has been used to an inverse scattering problem
related to the loss-dominated fractional Helmholtz equation.
In Section 5, we provide a short summary and propose a few further questions.

% ---------------------------------------------------------------------------------------------------------------------
%  Notations
% ---------------------------------------------------------------------------------------------------------------------
\section{Preliminaries}\label{p2Pre}

\subsection{Notations}

In this section, we provide an explanation of the notations used in the rest of this paper.
\begin{itemize}
  \item Let $n \in \mathbb{N}$ be an integer, and
  $\mathbb{R}^{n}$ denotes $n$-dimensional Euclidean space; as usual, $\mathbb{R}$ means $\mathbb{R}^{1}$.
  \item $\Gamma(\cdot)$ denotes the usual Gamma function, and the reader may find a good introduction in \cite{Podlubny}.
  \item For $s\in \mathbb{R}, p\in [1,\infty)$, and a bounded domain $D \subset \mathbb{R}^{n}$,
  $W^{s,p}(D)$ denotes the Sobolev space, which roughly means that the $s$ order weak derivative
  of a function belongs to the space $L^{p}(D)$. For brevity, we usually denote $W^{s,p}(D)$ as $H^{s}(D)$.
  \item The notation $\dot{H}^{s}(D)$ stands for the homogeneous Sobolev space \cite{Function1,Function2} which used widely in the
  PDE community.
  \item Let $D \subset \mathbb{R}^{n}$ be a bounded domain, then $C(D)$ denotes continuous functions and
  $C_{u}(D)$ denotes uniformly bounded continuous functions.
  \item $C$ usually denotes a general constant and may be different from line to line.
  \item We let $\mathcal{H}$ denotes a Hilbert space and $L^{+}(\mathcal{H})$ denotes the set of all symmetric, positive operators.
  $L_{1}^{+}(\mathcal{H})$ denotes the operators of trace class and belong to $L^{+}(\mathcal{H})$.
  \item Let $\mathcal{H}$ be a Hilbert space, for an operator $\mathcal{C} \in L_{1}^{+}(\mathcal{H})$,
  $\mathcal{N}(a,\mathcal{C})$ denotes a Gaussian measure on $\mathcal{H}$ with the mean $a\in \mathcal{H}$
  and the covariance \mbox{operator $\mathcal{C}$.}
  \item We let $\eta$ and $\epsilon$ be two random variables, with $\eta \perp \epsilon$ indicating that the two random variables are independent.
\end{itemize}

%\begin{lemma}\label{p2kongjianLemma}
%Let $\Omega \subset \mathbb{R}^{n}$ be a bounded domain with Lipschitz continuous boundary $\partial\Omega$.
%Let $V_{\Omega} := \int_{\mathbb{R}^{n}\backslash \Omega} \frac{dy}{|x-y|^{n+2s}}, \, x\in \Omega$.
%Then the following assertions hold.
%\begin{itemize}
%  \item There exists a constant $C = C(n,s) > 0$ such that
%  \begin{align*}
%  V_{\Omega}(x) = \int_{\mathbb{R}^{n}\backslash \Omega} \frac{dy}{|x-y|^{n+2s}} \geq C |\Omega|^{-2s/n},\quad \forall \, x\in \Omega.
%  \end{align*}
%  \item There exist some constants $C_{1},C_{2} = C(n,s,\Omega) > 0$ such that for every $u \in H_{0}^{s}(\Omega)$,
%  \begin{align*}
%  \int_{\Omega}|u(x)|^{2}dx \leq C_{1}\int_{\Omega}|u(x)|^{2}V_{\Omega}(x)dx \leq C_{2}\|u\|_{H^{s}(\Omega)}^{2}.
%  \end{align*}
%\end{itemize}
%\end{lemma}

\subsection{Fractional Laplace operator}\label{p2FractionalSection}
In this part, we provide an elementary introduction to the fractional Laplace operator which used through all of this paper.
Let $0 < \alpha < 1$ and set
\begin{align*}
\mathcal{L}^{1}(\mathbb{R}^{n}) := \left\{ u \, :\, \mathbb{R}^{n} \rightarrow \mathbb{R} \,\, \text{measurable, }\,
\int_{\mathbb{R}^{n}} \frac{|u(x)|}{(1+|x|)^{n+2\alpha}}dx < \infty \right\}.
\end{align*}
For $u \in \mathcal{L}^{1}(\mathbb{R}^{n}),\, x\in \mathbb{R}^{n}$ and $\epsilon > 0$, we write
\begin{align*}
(-\Delta)_{\epsilon}^{\alpha}u(x) = C_{n,\alpha}\int_{\{ y\in \mathbb{R}^{n}, \, |y-x| > \epsilon \}}
\frac{u(x) - u(y)}{|x-y|^{n+2\alpha}} dy
\end{align*}
with
\begin{align}\label{p2constantnum}
C_{n,\alpha} = \frac{\alpha 2^{2\alpha}\Gamma(\frac{n+2\alpha}{2})}{\pi^{\frac{n}{2}}\Gamma(1-\alpha)},
\end{align}
where $\Gamma$ denotes the usual Gamma function.
The fractional Laplacian $(-\Delta)^{\alpha}u$ of the function $u$ is defined by the formula
\begin{align}\label{p2defination1}
(-\Delta)^{\alpha}u(x) = C_{n,\alpha} \text{P.V.} \, \int_{\mathbb{R}^{n}}\frac{u(x)-u(y)}{|x-y|^{n+2\alpha}}dy
= \lim_{\epsilon \downarrow 0} (-\Delta)^{\alpha}_{\epsilon}u(x), \quad x\in \mathbb{R}^{n}
\end{align}
provided that the limit exists \cite{Silvestre1}.
Except this definition, one can also define $(-\Delta)^{\alpha}$ by using the method of bilinear Dirichlet forms \cite{Guan1},
that is, $(-\Delta)^{\alpha}$ is the closed selfadjoint operator on $L^{2}(\mathbb{R}^{n})$ associated with the
bilinear symmetric closed form
\begin{align}\label{p2defination2}
\mathcal{E}(u,\varphi) = \frac{C_{n,\alpha}}{2} \int_{\mathbb{R}^{n}}\int_{\mathbb{R}^{n}}
\frac{(u(x)-u(y))(\varphi(x)-\varphi(y))}{|x-y|^{n+2\alpha}}\,dx dy, \quad u,\varphi \in H^{\alpha}(\mathbb{R}^{n}),
\end{align}
in the sense that
\begin{align*}
D((-\Delta)^{\alpha}) = \{ u \in H^{\alpha}(\mathbb{R}^{n}), \, (-\Delta)^{\alpha}u \in L^{2}(\mathbb{R}^{n}) \}
\end{align*}
and
\begin{align*}
\mathcal{E}(u,\varphi) = ((-\Delta)^{\alpha}u, \, \varphi) = \int_{\mathbb{R}^{n}}\varphi (-\Delta)^{\alpha}u \,dx,
\quad
\forall \, u\in D((-\Delta)^{\alpha}),\, \varphi \in H^{\alpha}(\mathbb{R}^{n}).
\end{align*}
Actually, there are at least ten equivalent definitions about fractional Laplace operator and
the equivalence has been proved in an interesting paper \cite{TenDef}.
Since in Section \ref{2forwardsection}, we may need to face fractional elliptic equations in bounded domain, here,
we present the definition of regional fractional Laplacian \cite{Guan2}. Let $\Omega$ be a bounded domain in $\mathbb{R}^{n}$,
denote by $\mathcal{L}^{1}(\Omega)$ all the measurable function $u$ on $\Omega$ such that
$\int_{\Omega}\frac{|u(x)|}{(1+|x|)^{n+2\alpha}}dx < \infty$.
For $u \in \mathcal{L}^{1}(\Omega)$, $x \in \Omega$ and $\epsilon > 0$, we write
\begin{align}\label{p2definitionRegional}
A_{\Omega,\epsilon}^{\alpha}u(x) = C_{n,\alpha}\int_{y\in \Omega, |y-x|>\epsilon}
\frac{u(y) - y(x)}{|x-y|^{n+2\alpha}}\,dy,
\end{align}
where $C_{n,\alpha}$ defined as in (\ref{p2constantnum}).
\begin{definition}\label{p2regionalLaplace}
Let $u \in \mathcal{L}^{1}(\Omega)$. The regional fractional Laplacian $A_{\Omega}^{\alpha}$ is defined by the formula
\begin{align}\label{p2regionalDefformula}
A_{\Omega}^{\alpha}u(x) = \lim_{\epsilon \downarrow 0}A_{\Omega,\epsilon}^{\alpha}u(x), \quad x\in \Omega,
\end{align}
provided the limit exists.
\end{definition}
When $\Omega = \mathbb{R}^{n}$, $A_{\mathbb{R}^{n}}^{\alpha}$ is the fractional power of Laplacian defined in (\ref{p2defination1}).
In order to give the Gauss-Green formula in the fractional Laplace operator setting, we give the following definition \cite{Guan2}.
\begin{definition}\label{p2fangxiangdao}
For $0 \leq s < 2$, $u \in C^{1}(\Omega)$ and $z\in \partial\Omega$, we define the operator $\mathcal{N}^{s}$ on $\partial \Omega$
by
\begin{align}\label{p2fangxiangdaoFormula}
\mathcal{N}^{s}u(z) := -\lim_{t\downarrow 0}\frac{du(z-t\mathbf{n}(z))}{dt}t^{s},
\end{align}
provided that the limit exists. Here, $\mathbf{n}(z)$ denotes the outward normal vector of $\partial\Omega$ at the point
$z \in \partial\Omega$.
\end{definition}
Let $\rho(x) := \text{dist}(x,\partial\Omega) = \inf\{ |y-x|\,:\, y \in \partial\Omega \}, \quad x\in \Omega$,
and for a real number $\delta > 0$, we set $\Omega_{\delta} := \{x\in \Omega,:\, 0< \rho(x) <\delta\}$.
Let $\beta > 0$ be a real number, define
\begin{align}\label{2pdefH}
\begin{split}
h_{\beta}(x) = \left\{
\begin{aligned}
& \rho(x)^{\beta - 1}\quad \forall \, x\in\Omega_{\delta}, \, \beta\in (0,1)\cup (1,\infty),    \\
& \ln (\rho(x)) \quad \forall\, x\in \Omega_{\delta}, \, \beta = 1.
\end{aligned}
\right.
\end{split}
\end{align}
For $1 < \beta \leq 2$, we define the space
\begin{align}\label{p2defCspace}
C_{\beta}^{2}(\bar{\Omega}) := \{
u \, : \, u(x) = f(x)h_{\beta}(x) + g(x), \quad \forall\, x\in\Omega\, \text{for some} \, f,g \in C^{2}(\bar{\Omega})
\}.
\end{align}
The above function space has many good properties, for us, we need to use the following property which has been proved in \cite{Warma1}.
\begin{lemma}\label{p2propertiesoffun}
Let $1/2 < \beta <1$ and $u \in C_{2\beta}^{2}(\bar{\Omega})$. Then $u \in H^{\beta}(\Omega)$ and $A_{\Omega}^{\beta}u \in L^{p}(\Omega)$
for every $p \in [1,\infty)$.
\end{lemma}
Having these preparations, now, we can state the following fractional Gauss-Green formula \cite{Guan2,Warma1}.
\begin{lemma}\label{p2GaussGreenTheorem}
Let $1/2 < \alpha < 1$ and let $A_{\Omega}^{\alpha}$ be the nonlocal operator defined in Definition \ref{p2regionalLaplace}.
Then, for every $u := fh_{2\alpha} + g \in C_{2\alpha}^{2}(\bar{\Omega})$ and $\varphi \in C_{2\alpha}^{2}(\bar{\Omega})$,
\begin{align}\label{p2GaussGreenFormula}
\begin{split}
\int_{\Omega}A_{\Omega}^{\alpha}u(x)\varphi(x)dx = &
\frac{1}{2}C_{n,\alpha}\int_{\Omega}\int_{\Omega}\frac{(u(x)-u(y))(\varphi(x)-\varphi(y))}{|x-y|^{n+2\alpha}}\,dxdy \\
& - B_{n,\alpha}\int_{\partial\Omega}\varphi\mathcal{N}^{2-2\alpha}u\,dS,
\end{split}
\end{align}
where $dS$ denotes the surface measure, $B_{n,\alpha}$ is a constant related to $C_{n,s}$ which can be found in \cite{Guan2}
or \cite{Warma1}.
\end{lemma}
In the rest of this paper, $A_{\Omega}^{0}$ will be understood as the identity operator.

% ---------------------------------------------------------------------------------------------------------------------
%  Forward part
% ---------------------------------------------------------------------------------------------------------------------
\section{Forward Problem}\label{2forwardsection}

In this section, we attempt to construct well-posedness theory for the loss-dominated fractional Helmholtz equation
and the dispersion-dominated fractional Helmholtz equation.
Before going further, let us make more specific assumptions about these two equations and the following assumptions
are valid in all of the rest parts.

\textbf{Assumption 1}:
\begin{enumerate}
  \item In order to make our presentation more concisely, without loss of generality, we may assume the space dimension $n = 2$.
  \item $q(\cdot)$ is assumed to be a bounded function and has compact support. Denote $B_{R}$ as a ball centered at the original, then there exist $R > 0$ such that
        $\text{supp} (q) \subset B_{R}$. In addition, we assume that there exists two constant $q_{\text{min}},q_{\text{max}}$ such that
        $-1 < q_{\text{min}} \leq q(\cdot) \leq q_{\text{max}} < \infty$.
  \item $\gamma$ is a piecewise constant function, and without loss of generality, in this paper we assume
        $\gamma(x) = \tilde{\gamma} 1_{\Omega}$, where $\Omega$ is a subset of $B_{R}$ ($\bar{\Omega} \subsetneqq \supp(q) \subsetneqq B_{R}$)
        and $\tilde{\gamma}$ is a constant in $[0,1/2]$.
  \item $\eta, \tau$ are assumed to be two non-negative piecewise constant functions related to $\gamma$.
        Let $\tilde{\eta}, \tilde{\tau}$ be to two positive constants,
        $\tau(x) = \tilde{\tau}$ if $\gamma(x)\neq 0$ and $\tau(x) = 0$ if $\gamma(x) = 0$.
        $\eta(x) = \tilde{\eta}$ if $\gamma(x) \neq 1/2$ and $\eta(x) = 0$ if $\gamma(x) = 1/2$.
\end{enumerate}
\begin{remark}
All the assumptions in Assumption 1 are based on the physical model. For example, if we assume $\gamma(x) = 0$ in (\ref{1coeff}),
then $\tau(x) = 0$ for $\sin(\pi \gamma(x)) = 0$ and if we assume $\gamma(x) = 1/2$ in (\ref{1coeff}),
then $\eta(x) = 0$ for $\cos(\pi \gamma(x)) = 0$.
\end{remark}

Figure \ref{ShiYiTu1} presents the assumptions stated in Assumption 1 for
the relation between the area with attenuate media $\Omega$,
the support of the scatterer $\text{supp}(q)$ and the circle with radius $R$ clearly.
\begin{figure}[htbp]
\centering
\includegraphics[width=0.5\textwidth]{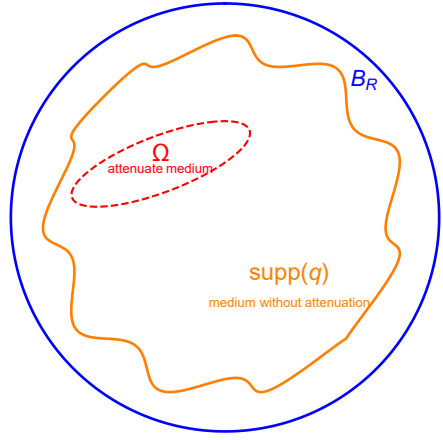}
\caption{The relation between the area with attenuate media $\Omega$,
the support of the scatterer $\text{supp}(q)$ and the circle with radius $R$.}\label{ShiYiTu1}
\end{figure}

Since one advantage of space fractional wave equation is that it can separate amplitude loss effect and dispersion effect,
we could study loss-dominated equation and dispersion-dominated equation separately.

\subsection{Loss-dominated model}\label{2.1lossmain}

In this subsection, we focus on the loss-dominated model. Based on the time-domain equation (\ref{1amplitude}),
we can easily derive the loss-dominated fractional Helmholtz equation as follows:
\begin{align}\label{2.1model}
\Delta u + i\omega\tau (-\Delta)^{\gamma+1/2}u + k^{2}(1+q(x))u = 0 \quad \text{in }\mathbb{R}^{2}.
\end{align}
As usual the scatterer is illuminated by a plane incident field
\begin{align}\label{2.1incident}
u^{\text{inc}}(x) = e^{i k x \cdot \mathbf{d}},
\end{align}
where $\mathbf{d} = (\cos(\theta), \sin(\theta)) \in \mathbb{S}^{1} = \{ x\in \mathbb{R}^{2}: |x| = 1 \}$ is
the incident direction and $\theta \in (0,2\pi)$ is the incident angle. Evidently, the incident field satisfies
\begin{align}\label{2.1incidentEquation}
\Delta u^{\text{inc}} + k^{2}u^{\text{inc}} = 0 \quad \text{in }\mathbb{R}^{2}.
\end{align}
Before going to set up the scattering problem, we need the following formula:
\begin{align}\label{2.1fractionalE}
(-\Delta)^{\alpha}e^{ikx\cdot\mathbf{d}} = k^{2\alpha}e^{ikx\cdot\mathbf{d}} \quad \text{with  }0<\alpha<1.
\end{align}
%\begin{proof}
%The proof of formula (\ref{2.1fractionalE}) relies on the fourier transform.
%In the following, we denote $A \approx B$ as $c B \leq A \leq C B$ with $c,C$ are general constants.
%Since
%\begin{align*}
%\int_{\mathbb{R}^{2}} e^{-i\xi\cdot x}(-\Delta)^{\alpha}e^{ik x\cdot \mathbf{d}} & \approx |\xi|^{2\alpha} \int_{\mathbb{R}^{2}}
%e^{-i\xi\cdot x} e^{ikx\cdot\mathbf{d}} dx   \approx |\xi|^{2\alpha}\delta(\xi - k\mathbf{d}),
%\end{align*}
%then we have
%\begin{align*}
%\int_{\mathbb{R}^{2}} e^{i\xi\cdot x} |\xi|^{2\alpha}\delta(\xi - k\mathbf{d})d\xi & \approx k^{2\alpha} e^{ikx\cdot\mathbf{d}}.
%\end{align*}
%Hence, formula (\ref{2.1fractionalE}) holds obviously.
%\end{proof}

The total field $u$ consists of the incident field $u^{\text{inc}}$ and the scattered field $u^{s}$:
\begin{align}\label{2.1inands}
u = u^{\text{inc}} + u^{s}.
\end{align}
It follows form (\ref{2.1model}), (\ref{2.1incidentEquation}), (\ref{2.1inands}) and formula (\ref{2.1fractionalE})
that the scattered field satisfies
\begin{align}\label{2.1scatterEquation}
\Delta u^{s} + i\omega\tau(-\Delta)^{\gamma + 1/2}u^{s} + k^{2}(1+q(x))u^{s} =
(-k^{2}q(x) - i\omega \tau k^{2\gamma + 1})u^{\text{inc}}
\end{align}
in $\mathbb{R}^{2}$.
By our assumption, function $\gamma$ is zero outside $\Omega$ which is contained in a ball with radius $R$, so
the scattered field as usual should satisfy the Sommerfeld radiation condition:
\begin{align}\label{2.1radiation}
\lim_{r\rightarrow \infty} \sqrt{r}\left(\partial_{r}u^{s} - iku^{s}\right) = 0,
\end{align}
where $r = |x|$.

In the domain $\mathbb{R}^{2}\backslash\bar{\Omega}$, equation (\ref{2.1scatterEquation}) reduced to
\begin{align}\label{2.1tradi}
\Delta u^{s} + k^{2}u^{s} = -k^{2}qu^{\text{inc}}.
\end{align}
This is just the equation in the classical scattering theory, so we know that
the solution of equation (\ref{2.1scatterEquation}) in $\mathbb{R}^{2}\backslash\bar{B}_{R}$ can be written under the
polar coordinates as follows:
\begin{align}\label{2.1series}
u^{s}(r,\theta) = \sum_{n \in \mathbb{Z}} \frac{H^{(1)}_{n}(kr)}{H^{(1)}_{n}(kR)}\hat{u}^{s}_{n}e^{in\theta},
\end{align}
where $H^{(1)}_{n}$ is the Hankel function of the first kind with order $n$ and
\begin{align*}
\hat{u}^{s}_{n} = (2\pi)^{-1}\int_{0}^{2\pi} u^{s}(R,\theta) e^{-in\theta} d\theta.
\end{align*}

Let $\mathbb{B}:H^{1/2}(\partial B_{R}) \rightarrow H^{-1/2}(\partial B_{R})$ be the Dirichlet-to-Neumann (DtN) operator defined as
follows: for an $u^{s} \in H^{1/2}(\partial B_{R})$,
\begin{align}\label{2.1Boundary}
(\mathbb{B}u^{s})(R,\theta) = k \sum_{n \in \mathbb{Z}} \frac{H^{(1)'}_{n}(kR)}{H^{(1)}_{n}(kR)}\hat{u}^{s}_{n}e^{in\theta}.
\end{align}
Using the DtN operator, the solution in (\ref{2.1series}) satisfies the following transparent boundary condition
\begin{align}\label{2.1Neumann}
\partial_{\mathbf{n}}u^{s} = \mathbb{B}u^{s} \quad \text{on }\partial B_{R},
\end{align}
where $\mathbf{n}$ is the unit outward normal on $\partial B_{R}$.
Now the problem can be converted to bounded domain.
Since we consider the bounded domain problem, the fractional Laplace operator may need to be adapted to
the regional fractional Laplace operator introduced in Section \ref{p2FractionalSection}.
Remembering the Assumption 1, for clarity, we write the bounded elliptic problem as follow
\begin{align}\label{2boundedFraPro}
\left\{
\begin{aligned}
& \Delta u^{s} + i\omega\tau A_{\Omega}^{\tilde{\gamma}+1/2}u^{s} + k^{2}(1+q)u^{s} = (-k^{2}q-i\omega\tau k^{2\gamma+1})u^{\text{inc}} \quad \text{in }B_{R}   \\
& \partial_{\mathbf{n}}u^{s} = \mathbb{B}u^{s} \quad \text{on }\partial B_{R}.
\end{aligned}
\right.
\end{align}
Now, the key step is that how to set up the weak formulation of the above problem (\ref{2boundedFraPro}),
for a good formulation will make our analysis simple.
Since there is Laplace operator in equation (\ref{2boundedFraPro}), we may expect that the solution $u$
belongs to $H^{1}(B_{R})$. Hence, for $z \in \partial \Omega$, we may have
\begin{align}\label{2fracBan}
\mathcal{N}^{2-2(\tilde{\gamma}+1/2)}u(z) = \mathcal{N}^{1-2\tilde{\gamma}}u(z)
= \lim_{t\downarrow 0}\mathbf{n(z)}\cdot\nabla u(z - t\mathbf{n}(z)) t^{1-2\tilde{\gamma}} = 0.
\end{align}
Based on this consideration, $A_{\Omega}^{\tilde{\gamma}+1/2}$ may be more appropriately be
defined as an operator with fractional Neumann boundary condition.
Inspired by the method used in \cite{Guan1,Guan2,Guan3,Warma1},
similar to the bilinear closed form defined in (\ref{p2defination2}), we need to consider the bilinear
closed form $\mathcal{E}_{\Omega}$ with domain $D(\mathcal{E}_{\Omega}) = H^{\tilde{\gamma}+1/2}(\Omega)$
and given for $u,\varphi \in H^{\tilde{\gamma}+1/2}(\Omega)$ by
\begin{align}\label{2.1bilinearBound}
\mathcal{E}_{\Omega}^{N}(u,\varphi) = \frac{C_{2, \tilde{\gamma}+1/2}}{2}\int_{\Omega}\int_{\Omega}
\frac{(u(x)-u(y))(\varphi(x)-\varphi(y))}{|x-y|^{3+2\tilde{\gamma}}}\,dxdy.
\end{align}
Let $A_{L}$ be the closed linear operator associated with the closed elliptic form $\mathcal{E}_{\Omega}^{N}$ in the sense that
\begin{align}\label{2.1NeumannOperator}
\left\{
\begin{aligned}
& D(A_{L}) := \{ u \in H^{\tilde{\gamma}+1/2}(\Omega),\, \exists\, v\in L^{2}(\Omega),    \\
& \quad\quad\quad\quad\quad \mathcal{E}_{\Omega}^{N}(u,\varphi) = (v,\varphi),\quad \forall \,\varphi \in H^{\tilde{\gamma}+1/2}(\Omega) \} \\
& A_{L} u = v.
\end{aligned}
\right.
\end{align}
\begin{remark}
The operator $A_{L}$ could be considered as a realization of the operator $A_{\Omega}^{\tilde{\gamma}+1/2}$ on
$L^{2}(\Omega)$ with fractional Neumann type boundary condition $\mathcal{N}^{2-2(\tilde{\gamma}+1/2)}u = 0$
on $\partial\Omega$. More precisely, if $\Omega$ has a $C^{2}$ boundary, we have the following
\begin{align}\label{2.1explanationA}
D(A_{L}) \cup C_{2\tilde{\gamma}+1}^{2}(\bar{\Omega}) = \{ u\in C_{2\tilde{\gamma}+1}^{2}(\bar{\Omega}),\,
\mathcal{N}^{1-2\tilde{\gamma}}u(z) = 0 \,\, \text{on}\,\, \partial\Omega \}.
\end{align}
\begin{proof}
Let
\begin{align*}
\mathcal{W} = & \{ u \in H^{\tilde{\gamma}+1/2}(\Omega) \cup C_{2\tilde{\gamma}+1}^{2}(\bar{\Omega}), \,
A_{\Omega}^{\tilde{\gamma}+1/2}u \in L^{2}(\Omega)  \\
& \quad\quad\quad\quad\quad\quad\quad\quad\quad\quad \text{ and } \mathcal{N}^{1-2\tilde{\gamma}}u = 0 \text{ weakly on }
\partial\Omega \}.
\end{align*}
Let $u \in D(A) \cup C_{2\tilde{\gamma}+1}^{2}(\bar{\Omega})$. Then by definition, there exists a
function $v \in L^{2}(\Omega)$ such that $\mathcal{E}_{\Omega}^{N}(u,\varphi) = (v,\varphi)_{L^{2}(\Omega)}$
for every $\varphi \in H^{\tilde{\gamma}+1/2}(\Omega)$.
Using the fractional Gauss-Green type formula, we have that for every $\varphi \in H^{\tilde{\gamma}+1/2}(\Omega)$,
\begin{align}\label{2.1equivalenceDef}
\begin{split}
\mathcal{E}_{\Omega}^{N}(u,\varphi) & = \frac{C_{2,\tilde{\gamma}+1/2}}{2}
\int_{\Omega}\int_{\Omega}\frac{(u(x)-u(y))(\varphi(x)-\varphi(y))}{|x-y|^{3+2\tilde{\gamma}}}\,dxdy    \\
& = \int_{\Omega}\varphi A_{\Omega}^{\tilde{\gamma}+1/2}u \,dx + B_{2,\tilde{\gamma}+1/2}
\int_{\partial\Omega}\varphi \mathcal{N}^{1-2\tilde{\gamma}}u \,dS  \\
& = \int_{\Omega} v\varphi \,dx.
\end{split}
\end{align}
It follows form (\ref{2.1equivalenceDef}) that in particular, for every $\varphi \in C_{c}^{\infty}(\Omega)$
\begin{align*}
\mathcal{E}_{\Omega}^{N}(u,\varphi) = \int_{\Omega}\varphi A_{\Omega}^{\tilde{\gamma}+1/2}u \,dx
= \int_{\Omega} \varphi v \,dx.
\end{align*}
Hence,
$$A_{L}u := v = A_{\Omega}^{\tilde{\gamma}+1/2}u.$$
Because $v\in L^{2}(\Omega)$, we obtain that $A_{\Omega}^{\tilde{\gamma}+1/2} \in L^{2}(\Omega)$.
Using the fact that $v = A_{\Omega}^{\tilde{\gamma}+1/2}u$, we obtain from (\ref{2.1equivalenceDef}) again that
$\mathcal{N}^{1-2\tilde{\gamma}}u = 0$ on $\partial\Omega$.
We have shown that $u \in \mathcal{W}$ and $A_{L}u = A_{\Omega}^{\tilde{\gamma}+1/2}u$.

Conversely, let $u \in \mathcal{W}$ and set $v := A_{\Omega}^{\tilde{\gamma}+1/2}u \in L^{2}(\Omega)$.
Then for every $\varphi \in H^{\tilde{\gamma}+1/2}(\Omega)$,
\begin{align*}
(v,\varphi)_{L^{2}(\Omega)} = \int_{\Omega}\varphi A_{\Omega}^{\tilde{\gamma}+1/2}u \,dx.
\end{align*}
Using the fractional Gauss-Green formula (Lemma \ref{p2GaussGreenFormula}) and the fact that
$\mathcal{N}^{1-2\tilde{\gamma}}u = 0$ on $\partial\Omega$, we obtain that for every $\varphi \in H^{\tilde{\gamma}+1/2}(\Omega)$,
\begin{align*}
(v,\varphi)_{L^{2}(\Omega)} = \frac{C_{2,\tilde{\gamma}+1/2}}{2}\int_{\Omega}\int_{\Omega}
\frac{(u(x)-u(y))(v(x)-v(y))}{|x-y|^{3+2\tilde{\gamma}}}\, dxdy = \mathcal{E}_{\Omega}^{N}(u,\varphi).
\end{align*}
We have shown that $u \in D(A_{L})$ and $A_{L} u = A_{\Omega}^{\tilde{\gamma}+1/2}u$.
Considering Lemma \ref{p2propertiesoffun}, we easily obtain that
\begin{align*}
\mathcal{W} = \{ u\in C_{2\tilde{\gamma}+1}^{2}(\bar{\Omega}),\, \mathcal{N}^{1-2\tilde{\gamma}}u = 0\,\text{on}\, \partial\Omega \}.
\end{align*}
Hence, the proof is completed.
\end{proof}
\end{remark}
Under these considerations, equation (\ref{2boundedFraPro}) should have the following form
\begin{align}\label{2boundedFraProFinal}
\left\{
\begin{aligned}
& \Delta u^{s} + i\omega\tau A_{L}u^{s} + k^{2}(1+q)u^{s} = (-k^{2}q-i\omega\tau k^{2\gamma+1})u^{\text{inc}} \quad \text{in }B_{R}   \\
& \partial_{n}u^{s} = \mathbb{B}u^{s} \quad \text{on }\partial B_{R}.
\end{aligned}
\right.
\end{align}
Define $a:\, H^{1}(B_{R}) \times H^{1}(B_{R}) \rightarrow \mathbb{C}$ as
\begin{align}\label{2.1definea}
\begin{split}
a(u^{s},\varphi) = & \int_{B_{R}} \nabla u^{s} \cdot \nabla \bar{\varphi} dx - i\omega\tau \mathcal{E}_{\Omega}^{N}(u^{s},\bar{\varphi})    \\
& \quad\quad -k^{2}\int_{B_{R}}(1+q(x))u^{s}\bar{\varphi} \,dx - \int_{\partial B_{R}}\mathbb{B}u^{s}\bar{\varphi}dS,
\end{split}
\end{align}
then define $b:\, H^{1}(B_{R}) \rightarrow \mathbb{C}$ as
\begin{align}\label{2.1defineb}
\begin{split}
b(\varphi) = \int_{B_{R}}(k^{2}q(x)+i\omega\tau k^{2\tilde{\gamma}+1})u^{\text{inc}}\bar{\varphi}\,dx.
\end{split}
\end{align}
Now by (\ref{2.1NeumannOperator}) and (\ref{2boundedFraProFinal}),
we easily obtain the variational form of equation (\ref{2boundedFraPro}) as follows:
\begin{align}\label{2.1variationalForm}
a(u^{s},\varphi) = b(\varphi) \quad \forall \, \varphi \in H^{1}(B_{R}).
\end{align}

For a given scatterer $q$, fractional order function $\gamma$ and an incident field $u^{\text{inc}}$, we define
the map $\mathfrak{S}(q,\gamma,u^{\text{inc}})$ by $u^{s} = \mathfrak{S}(q,\gamma,u^{\text{inc}})$,
where $u^{s}$ is the solution of the problem (\ref{2boundedFraPro}) or the variational problem (\ref{2.1variationalForm}).
It is easily seen that the map $\mathfrak{S}(q,\gamma,u^{\text{inc}})$ is linear with respect to $u^{\text{inc}}$ but is
nonlinear with respect to $q$, in addition, $\gamma$ is assumed to be known in the fractional scattering problem.
Hence, we may denote $\mathfrak{S}(q,\gamma,u^{\text{inc}})$ by $\mathfrak{S}(q)u^{\text{inc}}$.
Concerning the map $\mathfrak{S}(q)$, we have the following regularity result.
\begin{theorem}\label{2.1smallKexistence}
Let $0 < \tilde{\gamma} < 1/2$, if the wavenumber $k$ is sufficiently small, the variational problem (\ref{2.1variationalForm})
admits a unique weak solution in $H^{1}(B_{R})$ and $\mathfrak{S}(q)$ is a bounded linear map from $L^{2}(B_{R})$
to $H^{1}(B_{R})$. Furthermore, there is a constant $C$ depend on $B_{R}$ and $\|q\|_{L^{\infty}(B_{R})}$, such that
\begin{align}\label{2.1smallKFormula}
\|\mathfrak{S}(q)u^{\text{inc}}\|_{H^{1}(B_{R})} \leq C k \|u^{\text{inc}}\|_{L^{2}(B_{R})}.
\end{align}
\end{theorem}
The proof is inspired by the method used in \cite{Bao1,Bao4,Bao2,Bao3} for integer order Helmholtz equation, here, we give a sketch for concisely.
\begin{proof}
Define
\begin{align*}
a_{1}(u^{s},\varphi) & = (\nabla u^{s},\nabla \varphi) - i\omega\tau\mathcal{E}_{\Omega}^{N}(u^{s},\varphi)
- \langle \mathbb{B}u^{s},\varphi \rangle,    \\
a_{2}(u^{s},\varphi) & = -((1+q)u^{s},\varphi).
\end{align*}
It is obvious that $a = a_{1} + k^{2}a_{2}$.
Since
\begin{align*}
\mathcal{E}_{\Omega}^{N}(u^{s},u^{s}) = \frac{C_{2,s}}{2}\int_{\Omega}\int_{\Omega}
\frac{(u^{s}(x)-u^{s}(y))^{2}}{|x-y|^{3+2\tilde{\gamma}}}\,dxdy \geq 0,
\end{align*}
then, for $a_{1}$, we could obtain
\begin{align*}
|a_{1}(u^{s},u^{s})| & \geq C \|\nabla u^{s}\|_{L^{2}(B_{R})}^{2} + |\omega\tau| \mathcal{E}_{\Omega}^{N}(u^{s},u^{s})
+ C\|u^{s}\|_{L^{2}(\partial B_{R})}^{2}    \\
& \geq C \|u^{s}\|^{2}_{H^{1}(B_{R})},
\end{align*}
where we used Theorem 2.6.4 in \cite{BookBao1}.
Then we define an operator $\mathfrak{A}:L^{2}(B_{R}) \rightarrow H^{1}(B_{R})$ by
\begin{align*}
a_{1}(\mathfrak{A}u^{s},\varphi) = a_{2}(u^{s},\varphi), \quad \forall\, \varphi \in H^{1}(B_{R}).
\end{align*}
Using the Lax-Milgram lemma, it follows that
\begin{align}\label{2.1lax1}
\|\mathfrak{A}u^{s}\|_{H^{1}(B_{R})} \leq C \|u^{s}\|_{L^{2}(B_{R})}.
\end{align}
Define a function $w \in L^{2}(B_{R})$ by requiring $w \in H^{1}(B_{R})$ and satisfying
\begin{align}\label{2.1operator2}
a_{1}(w,\varphi) = b(\varphi) \quad \forall \, \varphi \in H^{1}(B_{R}).
\end{align}
It follows from the Lax-Milgram lemma again that
\begin{align}\label{2.1lax2}
\|w\|_{H^{1}(B_{R})} \leq C (k^{2}\|q\|_{L^{\infty}(B_{R})}+|\omega|\tau k^{2\tilde{\gamma}+1})\|u^{\text{inc}}\|_{L^{2}(B_{R})}.
\end{align}
Using the operator $\mathfrak{A}$, we can see that problem (\ref{2.1variationalForm}) is equivalent to find $u^{s} \in L^{2}(B_{R})$
such that
\begin{align}\label{2.1operatorEqu}
(I+k^{2}\mathfrak{A})u^{s} = w.
\end{align}
When the wavenumber is small enough, the operator $I+k^{2}\mathfrak{A}$ has a uniform bounded inverse.
Then we have the estimate $\|u^{s}\|_{L^{2}(B_{R})} \leq C \|w\|_{L^{2}(B_{R})}$.
Rearranging (\ref{2.1operatorEqu}), we have $u^{s} = w - k^{2}\mathfrak{A}u^{s}$, so we obtain
\begin{align*}
\|u^{s}\|_{H^{1}(B_{R})} \leq \|w\|_{H^{1}(B_{R})} + C k^{2} \|u^{s}\|_{L^{2}(B_{R})} \leq C k \|u^{\text{inc}}\|_{L^{2}(B_{R})},
\end{align*}
where we used (\ref{2.1lax2}) in the second inequality.
\end{proof}

In order to obtain a similar result for some general wavenumber $k > 0$, we need the following uniqueness result.
\begin{lemma}\label{2.1uniqueness}
Given the scatterer $q \in L^{\infty}(B_{R})$, the direct scattering problem (\ref{2boundedFraProFinal}) has at most
one solution.
\end{lemma}
\begin{proof}
It suffice to show that $u^{s} = 0$ in $B_{R}$ if $u^{\text{inc}} = 0$ (no source term). From the Green's formula
and fractional Gauss-Green formula (Lemma \ref{p2GaussGreenTheorem}), we have
\begin{align*}
0 & = \int_{B_{R}} u^{s} (\Delta \bar{u}^{s} + i\omega\tau A_{L}\bar{u}^{s}) - \bar{u}^{s}(\Delta u^{s}+i\omega\tau A_{L}u^{s}) \,dx    \\
& = \int_{\partial B_{R}} u^{s}\frac{\partial \bar{u}^{s}}{\partial \mathbf{n}} - \bar{u}^{s}\frac{\partial u^{s}}{\partial \mathbf{n}} dS
+ i\omega\tau \int_{\partial\Omega} u^{s}\mathcal{N}^{1-2\gamma}\bar{u}^{s} - \bar{u}^{s} \mathcal{N}^{1-2\gamma}u^{s} dS   \\
& = \int_{\partial B_{R}} u^{s}\frac{\partial \bar{u}^{s}}{\partial \mathbf{n}} - \bar{u}^{s}\frac{\partial u^{s}}{\partial \mathbf{n}} dS
= -2i \text{Im}\int_{\partial B_{R}} \bar{u}^{s} \mathbb{B}u^{s} \,dS.
\end{align*}
Now based on same ideas in the proof of Theorem 2.6.5 in \cite{BookBao1}, we obtain that $u^{s} = 0$ on $\partial B_{R}$.
The boundary condition (\ref{2.1Neumann}) yields further $\frac{\partial u^{s}}{\partial \mathbf{n}} = 0$ on $\partial B_{R}$.
Hence, we easily see that $u^{s} = 0$ in $\mathbb{R}^{2}\backslash B_{R}$.
Now, let us recall that for $u^{\text{inc}} = 0$, we have
\begin{align}\label{2boundedFraProzero}
\left\{
\begin{aligned}
& \Delta u^{s} + i\omega\tau A_{L}u^{s} + k^{2}(1+q)u^{s} = 0 \quad \text{in }B_{R}   \\
& \partial_{n}u^{s} = \mathbb{B}u^{s} \quad \text{on }\partial B_{R}.
\end{aligned}
\right.
\end{align}
Taking absolute value on both sides of the above equation, we obtain that
\begin{align*}
|\Delta u^{s}|^{2} + \omega^{2}\tau^{2}|A_{L}u^{s}|^{2} \leq |k^{2}(1+q)u^{s}| \quad \text{for }x\in B_{R}.
\end{align*}
Hence, it is obvious that
\begin{align*}
|\Delta u^{s}(x)|^{2} \leq |k^{2}(1+q(x))u^{s}(x)| \quad \text{for }x\in B_{R}.
\end{align*}
From the results in \cite{UniqueContinuation}, $u^{s} = 0$ in $B_{R}$.
\end{proof}
With the above lemma, we could obtain the following result for general $k > 0$ by using Fredholm alternative theorem.
\begin{theorem}\label{2.1generalK}
Given the scatterer $q \in L^{\infty}(B_{R})$, the variational problem (\ref{2.1variationalForm}) admits
a unique weak solution in $H^{1}(B_{R})$ for all $k > 0$ and $\mathfrak{S}(q)$ is a bounded linear map from
$L^{2}(B_{R})$ to $H^{1}(B_{R})$. Furthermore, the estimate
\begin{align}\label{2.1generalEsti}
\|\mathfrak{S}u^{\text{inc}}\|_{H^{1}(B_{R})} \leq C \|u^{\text{inc}}\|_{L^{2}(B_{R})},
\end{align}
holds, where the constant $C$ depends on $k$, $B_{R}$ and $\|q\|_{L^{\infty}(B_{R})}$.
\end{theorem}

\begin{theorem}\label{2.1LipCon}
Assume that $q_{1},q_{2} \in L^{\infty}(B_{R})$. Then
\begin{align}\label{2.1LipEst}
\|\mathfrak{S}(q_{1})u^{\text{inc}} - \mathfrak{S}(q_{2})u^{\text{inc}}\|_{H^{1}(B_{R})}
\leq C \|q_{1} - q_{2}\|_{L^{\infty}(B_{R})}\|u^{\text{inc}}\|_{L^{2}(B_{R})},
\end{align}
where the constant $C$ depends on $k$, $B_{R}$ and $\|q_{2}\|_{L^{\infty}(B_{R})}$.
\end{theorem}
\begin{proof}
Let $u_{1}^{s} = \mathfrak{S}(q_{1})u^{\text{inc}}$ and $u_{2}^{s} = \mathfrak{S}(q_{2})u^{\text{inc}}$.
It follows that for $j=1,2$
\begin{align*}
\Delta u_{j}^{s} + i\omega\tau A_{L}u_{j}^{s} + k^{2}(1+q_{j})u_{j}^{s} = (-k^{2}q_{j}-i\omega\tau k^{2\gamma+1})u^{\text{inc}}.
\end{align*}
By setting $\delta u^{s} = u_{1}^{s} - u_{2}^{s}$, we have
\begin{align*}
\Delta \delta u^{s} + i\omega\tau A_{L}\delta u^{s} + k^{2}(1+q_{1})\delta u^{s} = -k^{2}(q_{1}-q_{2})(u^{\text{inc}}+u_{2}^{s}).
\end{align*}
The function $\delta u^{s}$ also satisfies the boundary condition
\begin{align*}
\partial_{n}\delta u^{s} = \mathbb{B}\delta u^{s} \quad \text{on }\partial B_{R}.
\end{align*}

Using similar methods in proving Theorem \ref{2.1generalK}, we could obtain that
\begin{align*}
\|\delta u^{s}\|_{H^{1}(B_{R})} \leq C \|q_{1}-q_{2}\|_{L^{\infty}(B_{R})}\|u^{\text{inc}} + u_{2}^{s}\|_{L^{2}(B_{R})}.
\end{align*}
Using Theorem \ref{2.1generalK} for $u_{2}^{s}$, we have
\begin{align*}
\|u_{2}^{s}\|_{H^{1}(B_{R})} \leq C \|u^{\text{inc}}\|_{L^{2}(B_{R})},
\end{align*}
which gives
\begin{align*}
\|\mathfrak{S}(q_{1})u^{\text{inc}} - \mathfrak{S}(q_{2})u^{\text{inc}}\|_{H^{1}(B_{R})}
\leq C \|q_{1} - q_{2}\|_{L^{\infty}(B_{R})}\|u^{\text{inc}}\|_{L^{2}(B_{R})},
\end{align*}
where the constant $C$ depends on $k$, $B_{R}$ and $\|q_{2}\|_{L^{\infty}(B_{R})}$.
\end{proof}

\subsection{Dispersion-dominated model}
In this section, we focus on the dispersion-dominated model.
Based on time-domain equation (\ref{1dispersion}),
we can easily obtain the dispersion-dominated fractional Helmholtz equation as follows:
\begin{align}\label{2.2helmholtzEquation}
(-\Delta)^{\gamma+1}u - k^{2}(1+q)u = 0 \quad\text{in }\mathbb{R}^{2}.
\end{align}
It is obvious that model (\ref{2.2helmholtzEquation}) is a higher order elliptic equation, so we may
transfer it to a lower order elliptic system.
As in the above section, the total field $u$ consists of the incident field $u^{\text{inc}}$ and the scattered field $u^{s}$:
\begin{align}\label{2.2inands}
u = u^{\text{inc}} + u^{s},
\end{align}
with $u^{\text{inc}}(x) = e^{i k x \cdot \mathbf{d}}$. Using formula (\ref{2.1fractionalE}), we will obtain
\begin{align}\label{2.2incidentEquation}
(-\Delta)^{\gamma+1}u^{\text{inc}} - k^{2\gamma + 2}u^{\text{inc}} = 0.
\end{align}
Now, we easily obtain the scattered field $u^{s}$ satisfies
\begin{align}\label{2.2scatterEquation}
(-\Delta)^{\gamma+1}u^{s} - k^{2}(1+q)u^{s} = (k^{2}q + k^{2} - k^{2\gamma+2})u^{\text{inc}}.
\end{align}
Because (\ref{2.2scatterEquation}) is a $2\gamma+2 \geq 2$ order elliptic equation,
this equation seems more difficult than the loss-dominated equation.

By our assumption, there has attenuation effect in the domain $\Omega \subset \supp(q) \subset B_{R}$  and
no attenuation effect in $B_{R} \backslash \Omega$.
Hence, we may see that the operator $(-\Delta)^{\gamma}$ brings some ``perturbation'' of the non-attenuation equation and the higher order
equation (\ref{2.2helmholtzEquation}) could be transformed to the following form:
\begin{align}\label{2.2systemPDE}
\begin{split}
\left\{
\begin{aligned}
& (-\Delta)g^{s} - k(1+q)u^{s} = (k q+k-k^{2\gamma+1})u^{\text{inc}} \quad \text{in }\mathbb{R}^{2},  \\
& (-\Delta)^{\tilde{\gamma}}u^{s} - k g^{s} = 0 \quad \text{in }\Omega,    \\
& \lim_{r\rightarrow \infty}\sqrt{r}\left( \frac{\partial g^{s}}{\partial\mathbf{n}} - ikg^{s} \right) = 0, \\
& u^{s} = g^{s} \quad \text{on }\partial\Omega,
\end{aligned}
\right.
\end{split}
\end{align}
with $r = |x|$. In the above system and in the following, $u^{s}$ outside $\Omega$ understand as $g^{s}$.
Our fractional equation could be reduced to (\ref{2.1tradi}) on $\mathbb{R}^{2}\backslash \Omega$, hence,
the operator defined in (\ref{2.1Boundary}) still valid. And as considered in the loss-dominated case, we consider
bounded domain equation, hence, we may replace $(-\Delta)^{\tilde{\gamma}}$ to $A_{\Omega}^{\tilde{\gamma}}$.
Based on these considerations, we obtain the following elliptic system
\begin{align}\label{2.2systemPDEFinal}
\begin{split}
\left\{
\begin{aligned}
& (-\Delta)g^{s} - k(1+q)u^{s} = (k q+k-k^{2\gamma+1})u^{\text{inc}} \quad \text{in }B_{R},  \\
& A_{\Omega}^{\tilde{\gamma}}u^{s} - k g^{s} = 0 \quad \text{in }\Omega,    \\
& \frac{\partial}{\partial\mathbf{n}}g^{s} = \mathbb{B}g^{s}  \quad \text{on }\partial B_{R}, \\
& u^{s} = g^{s} \quad \text{on } \partial\Omega.
\end{aligned}
\right.
\end{split}
\end{align}
Because $g^{s}$ satisfies a second order elliptic equation, we may expect that $g^{s} \in H^{1}(B_{R})$.
From the second equation in system (\ref{2.2systemPDEFinal}), we may expect that $u^{s}\in H^{1}(B_{R})$.
Hence, there should be no boundary term in the fractional Gauss-Green formula (\ref{p2GaussGreenFormula}).
Define
\begin{align*}
\mathcal{E}_{\Omega}^{D}(u^{s},\psi) = \frac{C_{2,\tilde{\gamma}}}{2} \int_{\Omega}\int_{\Omega}
\frac{(u^{s}(x)-u^{s}(y))(\psi(x)-\psi(y))}{|x-y|^{2+2\tilde{\gamma}}}dxdy.
\end{align*}
For $U^{s} = (g^{s},u^{s}) \in H^{1}(B_{R}) \times H^{\tilde{\gamma}}(\Omega)$
and $\Phi = (\varphi,\psi) \in H^{1}(B_{R}) \times H^{\tilde{\gamma}}(\Omega)$, define
\begin{align*}
a_{D}(U^{s},\Phi) = & \int_{B_{R}}\nabla g^{s}\cdot \nabla \bar{\varphi} dx - \int_{\partial B_{R}}\mathbb{B}g^{s}\bar{\varphi}dS \\
& - k \int_{B_{R}}(1+q)u^{s}\bar{\varphi}dx + \mathcal{E}_{\Omega}^{D}(u^{s},\bar{\psi})
- k\int_{\Omega}g^{s}\bar{\psi}dx,
\end{align*}
and
\begin{align*}
b_{D}(\Phi) = k\int_{B_{R}}(q + 1 - k^{2\gamma})u^{\text{inc}}\bar{\varphi}dx.
\end{align*}
Then we can define the weak formulation of system (\ref{2.2systemPDEFinal}) as follow:
\begin{align}\label{2.2weakFormula}
a_{D}(U^{s},\Phi) = b_{D}(\Phi) \quad \forall \, \Phi \in H^{1}(B_{R}) \times H^{\tilde{\gamma}}(\Omega).
\end{align}

Now we give the main result of this section.
\begin{theorem}\label{2.2smallK}
Let $0 < \tilde{\gamma} < 1/2$, for a large enough constant $C > 0$,
\begin{align*}
C k^{1/2}(1+\|q\|_{L^{\infty}(B_{R})}) < 1.
\end{align*}
The variational problem (\ref{2.2weakFormula})
admits a unique weak solution in $H^{1}(B_{R}) \times \dot{H}^{\tilde{\gamma}}(\Omega)$.
\end{theorem}
\begin{proof}
\textbf{Step 1}:
Because the complexity of our problem, we choose an iterative methods to show the existence of this problem.
Let $u_{0}^{s} = g_{0}^{s} = 0$, then we can write the following system
\begin{align}\label{2.2iterative}
\left\{
\begin{aligned}
& -\Delta g_{n+1}^{s} = k(1+q)u_{n}^{s} + k(q+1-k^{2\gamma})u^{\text{inc}} \quad \text{in }B_{R},   \\
& A_{\Omega}^{\tilde{\gamma}}u_{n+1}^{s} = kg_{n}^{s} \quad\text{in }\Omega,  \\
& \frac{\partial}{\partial\mathbf{n}}g^{s}_{n+1} = \mathbb{B}g_{n+1}^{s} \quad \text{on }\partial B_{R},    \\
& u_{n+1}^{s} = g_{n}^{s} \quad \text{on }\partial\Omega.
\end{aligned}
\right.
\end{align}
The weak form of the above system (\ref{2.2iterative}) then could be written as follows:
\begin{align}\label{2.2iterativeWeak1}
\begin{split}
& \int_{B_{R}}\nabla g_{n+1}^{s}\cdot\nabla \bar{\varphi} dx - \int_{\partial B_{R}}\mathbb{B}g_{n+1}^{s}\bar{\varphi} dS \\
& \quad\quad\quad\quad\quad\quad\quad
= k\int_{B_{R}}(1+q)u_{n}^{s}\bar{\varphi} dx + k\int_{B_{R}}(q+1-k^{2\gamma})u^{\text{inc}}\bar{\varphi} dx,
\end{split}
\end{align}
\begin{align}\label{2.2iterativeWeak2}
\begin{split}
& \mathcal{E}_{\Omega}^{D}(u_{n+1}^{s},\bar{\psi}) + \int_{\partial\Omega}u_{n+1}^{s}\bar{\psi}dS \\
& \quad\quad\quad\quad\quad\quad\quad
= k\int_{\Omega}g_{n}^{s}\bar{\psi}dx + \int_{\partial\Omega}g_{n}^{s}\bar{\psi}dS.
\end{split}
\end{align}
Considering this system could be solved easily by using Lax-Milgram lemma,
we may obtain a series of solution $u^{s}_{n} \in H^{\tilde{\gamma}}(\Omega)$ and $g_{n}^{s} \in H^{1}(B_{R})$ with $n = 0,1,2,\cdots$.

Now, we need some uniform estimates of the solution series $\{g^{s}_{n},u_{n}^{s}\}_{n=0,1,2,\cdots}$.
Taking $\varphi$, $\psi$ equal to $g_{n+1}^{s}$, $u_{n+1}^{s}$ in (\ref{2.2iterativeWeak1}),
we will obtain
\begin{align}\label{2.2uniform1}
\begin{split}
\|g_{n+1}^{s}\|_{H^{1}(B_{R})} \leq & C k (1+\|q\|_{L^{\infty}(B_{R})})\|u_{n}^{s}-g_{n-1}^{s}\|_{L^{2}(B_{R})}    \\
& + C k (1+\|q\|_{L^{\infty}(B_{R})})\|g_{n-1}^{s}\|_{L^{2}(B_{R})}  \\
& + C k (1+\|q\|_{L^{\infty}(B_{R})}) \|u^{\text{inc}}\|_{L^{2}(B_{R})},    \\
\leq & C k (1+\|q\|_{L^{\infty}(B_{R})})\|u_{n}^{s}-g_{n-1}^{s}\|_{\dot{H}^{\tilde{\gamma}}(B_{R})}    \\
& + C k (1+\|q\|_{L^{\infty}(B_{R})})\|g_{n-1}^{s}\|_{L^{2}(B_{R})}  \\
& + C k ( 1+  \|q\|_{L^{\infty}(B_{R})}) \|u^{\text{inc}}\|_{L^{2}(B_{R})}.
\end{split}
\end{align}
For the second inequality, the results in \cite{Poincare} have been used, which are similar to the Poincar\'{e} inequalities.
By simple calculations, we have
\begin{align}\label{2.2guji1}
\begin{split}
\|g_{n+1}^{s}\|_{H^{1}(B_{R})} \leq & C k(1+\|q\|_{L^{\infty}(B_{R})})\|u_{n}^{s}\|_{\dot{H}^{\tilde{\gamma}}(B_{R})} \\
& + 2 C k(1+\|q\|_{L^{\infty}(B_{R})})\|g_{n-1}^{s}\|_{H^{1}(B_{R})} \\
& + C k (1+\|q\|_{L^{\infty}(B_{R})}) \|u^{\text{inc}}\|_{L^{2}(B_{R})}.
\end{split}
\end{align}
Taking $\varphi$, $\psi$ equal to $g_{n+1}^{s}$, $u_{n+1}^{s}$ in (\ref{2.2iterativeWeak2}), we will obtain that
\begin{align}\label{2.2guji2}
\begin{split}
\|u_{n+1}^{s}\|_{\dot{H}^{\tilde{\gamma}}(\Omega)}^{2} & \leq k \int_{\Omega}g_{n}^{s}(u_{n+1}^{s}-g_{n}^{s}) + |g_{n}^{s}|^{s}dx  \\
& \leq k \|g_{n}^{s}\|_{L^{2}(\Omega)}\|u_{n+1}^{s} - g_{n}^{s}\|_{\dot{H}^{\tilde{\gamma}}(\Omega)}
+ k \|g_{n}^{s}\|_{L^{2}(\Omega)}  \\
& \leq C k \|g_{n}^{s}\|_{H^{1}(\Omega)}^{2} + \frac{1}{2}\|u_{n+1}^{s}\|_{\dot{H}^{\tilde{\gamma}}(\Omega)}^{2}.
\end{split}
\end{align}
Then, we easily know that
\begin{align}\label{2.2guji3}
\|u_{n+1}^{s}\|_{\dot{H}^{\tilde{\gamma}}(\Omega)} \leq C k^{1/2} \|g_{n}^{s}\|_{H^{1}(B_{R})}.
\end{align}
Now we assume that $\|g_{k}^{s}\|_{H^{1}(B_{R})} + \|u_{k}^{s}\|_{\dot{H}^{\tilde{\gamma}}(\Omega)} \leq 1$ with $k = 0,1,\cdots,n-1$.
Combining (\ref{2.2guji1}) and (\ref{2.2guji3}), we finally obtain that
\begin{align}\label{2.2guji4}
\|g_{n}^{s}\|_{H^{1}(B_{R})} + \|u_{n}^{s}\|_{\dot{H}^{\tilde{\gamma}}(\Omega)}
\leq C k^{1/2}(1+\|q\|_{L^{\infty}(B_{R})})(1+\|u^{\text{inc}}\|_{L^{2}(B_{R})}).
\end{align}
By our condition on $k$, we know that
\begin{align}\label{2.2kongzhi}
\|g_{n}^{s}\|_{H^{1}(B_{R})} + \|u_{n}^{s}\|_{\dot{H}^{\tilde{\gamma}}(\Omega)} \leq 1.
\end{align}
Hence, we obtain that (\ref{2.2kongzhi}) holds for $n \in \mathbb{N}$.

From Section 7 in \cite{fractionalSobolev}, we know that $\dot{H}^{\tilde{\gamma}}(\Omega)$ and $H^{1}(B_{R})$ are compactly
embedded into the space $L^{2}(\Omega)$ and $L^{2}(B_{R})$ separately, then
we could obtain that for some function $u^{s}$ and $g^{s}$,
\begin{align}\label{2.2convergencePro}
\begin{split}
& u_{n}^{s} \rightharpoonup u^{s} \quad \text{in }\dot{H}^{\tilde{\gamma}}(\Omega), \quad
g_{n}^{s} \rightharpoonup g^{s} \quad \text{in }H^{1}(B_{R}),    \\
& u_{n}^{s} \rightarrow u^{s} \quad \text{in }L^{2}(\Omega), \quad
g_{n}^{s} \rightarrow g^{s} \quad \text{in }L^{2}(B_{R}),
\end{split}
\end{align}
where ``$\rightharpoonup$'' stands for weak convergence.
Adding (\ref{2.2iterativeWeak1}) and (\ref{2.2iterativeWeak2}) together, then using the above
convergence properties (\ref{2.2convergencePro}), we finally arrive at
\begin{align}\label{2.2yanzheng}
a_{D}(U^{s},\Phi) = b_{D}(\Phi),
\end{align}
with $U^{s} = (g^{s},u^{s})$ and $\Phi = (\varphi,\psi)$.
Hence, a solution of our system (\ref{2.2weakFormula}) has been found.\\
\textbf{Step 2}: Taking two solutions $U_{1}^{s} = (g_{1}^{s},u_{1}^{s})$ and $U_{2}^{s} = (g_{2}^{s},u_{2}^{s})$.
Denote $\delta U^{s} = U_{1}^{s} - U_{2}^{s} = (\delta g^{s}, \delta u^{s})$, then $\delta U^{s}$ satisfies
\begin{align}\label{2.2PDEFinalUniqueness}
\begin{split}
\left\{
\begin{aligned}
& (-\Delta)\delta g^{s} - k (1+q)\delta u^{s} = 0 \quad \text{in }B_{R},  \\
& A_{\Omega}^{\tilde{\gamma}}\delta u^{s} - k \delta g^{s} = 0 \quad \text{in }\Omega,    \\
& \frac{\partial}{\partial\mathbf{n}}\delta g^{s} = \mathbb{B}\delta g^{s}  \quad \text{on }\partial B_{R}, \\
& \delta u^{s} = \delta g^{s} \quad \text{on }\partial\Omega.
\end{aligned}
\right.
\end{split}
\end{align}
For the above system (\ref{2.2PDEFinalUniqueness}), performing same procedure
from (\ref{2.2guji1}) to (\ref{2.2guji4}), we could obtain that
\begin{align}\label{2.2uniqueness}
\begin{split}
& \|\delta g^{s}\|_{H^{1}(B_{R})} + \|\delta u^{s}\|_{\dot{H}^{\tilde{\gamma}}(\Omega)}   \\
& \quad\quad\quad
\leq C k^{1/2}(1+\|q\|_{L^{\infty}(B_{R})})(\|\delta g^{s}\|_{H^{1}(B_{R})} + \|\delta u^{s}\|_{\dot{H}^{\tilde{\gamma}}(\Omega)}).
\end{split}
\end{align}
From our assumptions, we find that
\begin{align}\label{2.2uniquenessFinal}
\|\delta g^{s}\|_{H^{1}(B_{R})} + \|\delta u^{s}\|_{\dot{H}^{\tilde{\gamma}}(\Omega)}
< \|\delta g^{s}\|_{H^{1}(B_{R})} + \|\delta u^{s}\|_{\dot{H}^{\tilde{\gamma}}(\Omega)}.
\end{align}
Hence, the proof is completed.
\end{proof}
\begin{remark}\label{2.2mainremark}
Theorem \ref{2.2smallK} seems strange, we only obtain $H^{\tilde{\gamma}}(\Omega)$ regularity of $u^{s}$.
From the second equation in (\ref{2.2systemPDEFinal}) and $g^{s} \in H^{1}(B_{R})$, we may obtain that
$u^{s}\in H^{1+\tilde{\gamma}}(\Omega) \cup H^{1}(B_{R})$.
The key point is the following fractional order elliptic equation
\begin{align}\label{2.2remarkHigh}
\left\{
\begin{aligned}
& A_{\Omega}^{\tilde{\gamma}}u^{s} = k g^{s} \quad\text{in }\Omega,  \\
& u^{s} = g^{s} \quad \text{on }\partial\Omega.
\end{aligned}
\right.
\end{align}
Intuitively, we could obtain higher regularity properties of $u^{s}$,
however, to our knowledge, there is no rigorous results about interior and boundary regularity for
equation (\ref{2.2remarkHigh}).
Because the studies about regularity properties of elliptic equation with operator $A_{\Omega}^{s}$ is
also a new topic in elliptic equation field \cite{Guan1,Guan2,MouChen},
a rigorous investigation of equation (\ref{2.2remarkHigh}) deserved to write another paper.
Hence, we would not investigate further on equation (\ref{2.2remarkHigh}) in this paper.
\end{remark}

\begin{remark}
Equation (\ref{2.2scatterEquation}) seems much more difficult than loss-dominated equation (\ref{2.1model}).
For general $k$, we can not provide a uniqueness result similar to Lemma \ref{2.1uniqueness}.
Different from the integer-order case \cite{UniqueContinuation}, 
there seems no unique continuation result of the fractional Laplace operator.
Hence, the dispersion-dominated equation needs further investigations.
\end{remark}

% ---------------------------------------------------------------------------------------------------------------------
%  Inverse part
% ---------------------------------------------------------------------------------------------------------------------
\section{Inverse Methods}\label{3inverse}

In the first part of this section, we provide the well-posedness theory of Bayesian inversion with model reduction error.
Then, as a straightforward extension, we show the relationship between the Bayesian method and the regularization method.
In the second part of this section, we investigate the small error limit problem, that is, whether the estimated function
tends to be the true function if both the model reduction error and white noise vanish. 
At last, the general theory has been applied to a concrete inverse scattering problem.

\subsection{Well-posedness}

Let $X,Y$ be separable Hilbert space, equipped with the Borel $\sigma$-algebra,
and $\mathcal{G}_{a}:X\rightarrow Y$ be a measurable mapping. We wish to solve the inverse problem of finding $x$ from $y$ where
\begin{align}\label{3.2approModel}
y = \mathcal{G}_{a}(x) + \epsilon + \eta,
\end{align}
and $\eta \in Y$ denotes noise, $\epsilon$ denotes model reduction error.
We employ a Bayesian approach to this problem in which we let $(x,y) \in X \times Y$ be a random variable and compute
$x|y$.
We specify the assumptions on the random variable $(x,y)$ as follows:

\textbf{Assumption 2}:
\begin{itemize}
  \item Prior: $x \sim \mu_{0}$ measure on $X$ and $\mu_{0}$ is chosen to be a Gaussian with mean $\bar{x}$ and
  covariance operator $\mathcal{C}_{x} \in L_{1}^{+}(X)$.
  \item Noise: $\eta \sim \mathbb{Q}_{0} = \mathcal{N}(0,\mathcal{C}_{\eta})$ measure on $Y$
  with $\mathcal{C}_{\eta} \in L_{1}^{+}(Y)$, and $\eta \perp x$.
  \item Model Reduction Error: $\epsilon \sim \mathbb{R}_{\bar{\epsilon}} = \mathcal{N}(\bar{\epsilon},\mathcal{C}_{\epsilon})$
  measure on $Y$ with $\mathcal{C}_{\epsilon} \in L_{1}^{+}(Y)$, and $\eta \perp \epsilon$.
\end{itemize}

For simplicity, we take $Y = \mathbb{R}^{J}$ with $J\in \mathbb{N}^{+}$ in the following.
Denote $(E, \langle\cdot,\cdot\rangle, \|\cdot\|_{E})$ as the Cameron-Martin space of the Gaussian measure $\mu_{0}$ on $X$, and we make
the following assumptions concerning the potential $\Phi$ appeared in the Baye's formula below.

\textbf{Assumption 3}: The function $\Phi : X\times Y \rightarrow \mathbb{R}$ satisfies the following:
\begin{enumerate}
  \item For every $\epsilon > 0$, there is an $M \in \mathbb{R}$, such that for all $u\in X$,
  \begin{align*}
  \Phi(x;y) \geq M - \epsilon \|u\|_{X}^{2}.
  \end{align*}
  \item there exist $p > 0$ and for every $r > 0$ a $K_{1} = K_{1}(r) > 0$ such that, for all $x \in X$ and $y \in Y$
  with $|y| < r$,
  \begin{align*}
  \Phi(x;y) \leq K_{1}(1+\|x\|_{X}^{p});
  \end{align*}
  \item for every $r > 0$ there is $K_{2} = K_{2}(r) > 0$ such that, for all $x_{1},x_{2} \in X$ and $y \in Y$ with
  $\max\{ \|x_{1}\|_{X}, \|x_{2}\|_{X}, |y| \} < r$,
  \begin{align*}
  |\Phi(x_{1};y) - \Phi(x_{2};y)| \leq K_{2}\|x_{1}-x_{2}\|_{X};
  \end{align*}
  \item there is $q \geq 0$ and for every $r > 0$ a $K_{3} = K_{3}(r) > 0$ such that, for all $y_{1},y_{2} \in Y$ with
  $\max\{ |y_{1}|, |y_{2}| \} < r$, and for all $x \ in X$,
  \begin{align*}
  |\Phi(x;y_{1}) - \Phi(x;y_{2})| \leq K_{3}(1+\|x\|_{X}^{q})|y_{1}-y_{2}|.
  \end{align*}
\end{enumerate}

Usually, we can not assume $\epsilon \perp x$ in (\ref{3.2approModel}), hence, we assume $(\epsilon,x) \in \mathcal{H}:= Y\times X$ distributed
according to a Gaussian measure $\mathcal{N}((\bar{\epsilon},\bar{x}),\mathcal{C})$.
Denote
\begin{align*}
& \mathcal{C}_{x} = \mathbb{E}(x-\bar{x})\otimes (x-\bar{x}), \quad
\mathcal{C}_{\epsilon} = \mathbb{E}(\epsilon - \bar{\epsilon}) \otimes (\epsilon - \bar{\epsilon}), \\
& \mathcal{C}_{x\epsilon} = \mathbb{E}(x-\bar{x})\otimes (\epsilon-\bar{\epsilon}), \quad
\mathcal{C}_{\epsilon x} = \mathbb{E}(\epsilon-\bar{\epsilon}) \otimes (x-\bar{x}).
\end{align*}
According to Theorem 6.20 in \cite{Bayesian3} (the results may also be found in early study \cite{BayesianEarly}),
we find that $\epsilon | x \sim \mathcal{N}(\bar{\epsilon}_{x},\mathcal{C}_{\epsilon |x})$ where
\begin{align*}
& \bar{\epsilon}_{x} = \bar{\epsilon} + \mathcal{C}_{\epsilon x}\mathcal{C}_{x}^{-1}(x-\bar{x}),    \\
& \mathcal{C}_{\epsilon |x} = \mathcal{C}_{\epsilon} - \mathcal{C}_{\epsilon x}\mathcal{C}_{x}^{-1}\mathcal{C}_{x\epsilon}.
\end{align*}
Define
\begin{align}\label{3.2nota1}
\nu = \epsilon + \eta,
\end{align}
then we have
\begin{align}\label{3.2nota2}
\nu | x = \epsilon | x + \eta,
\end{align}
and $\nu | x \sim \mathcal{N}(\bar{\nu}_{x},\mathcal{C}_{\nu | x})$ where
\begin{align}\label{wuNgongshi}
\begin{split}
& \bar{\nu}_{x} = \bar{\epsilon} + \mathcal{C}_{\epsilon x}\mathcal{C}_{x}^{-1}(x-\bar{x}),   \\
& \mathcal{C}_{\nu | x} = \mathcal{C}_{\eta} + \mathcal{C}_{\epsilon} - \mathcal{C}_{\epsilon x}\mathcal{C}_{x}^{-1}\mathcal{C}_{x\epsilon}.
\end{split}
\end{align}
Thus, we obtain that
\begin{align*}
y|x \sim \tilde{\mathbb{Q}} := \mathcal{N}(\mathcal{G}(x) + \bar{\nu}_{x}, \mathcal{C}_{\nu | x}).
\end{align*}
We assume through out the following that $\tilde{\mathbb{Q}} \ll \mathbb{Q}_{0}$ for $x$ $\mu_{0}$-a.s.
Thus, for some potential $\Phi : X\times Y \rightarrow \mathbb{R}$,
\begin{align}\label{3.2potential}
\frac{d\tilde{\mathbb{Q}}}{d\mathbb{Q}_{0}}(y) = \frac{1}{Z(y)}\exp(-\Phi(x;y)).
\end{align}
Thus, for fixed $x$, $\Phi (x;\cdot) : Y \rightarrow \mathbb{R}$ is measurable.

Define $\varsigma_{0}$ to be the product measure
\begin{align}\label{3.2productOri}
\varsigma_{0}(dx,dy) = \mu_{0}(dx)\mathbb{Q}_{0}(dy).
\end{align}
We assume in what follows that $\Phi(\cdot,\cdot)$ is $\varsigma_{0}$ measurable. Then the random variable
$(x,y)$ is distributed according to measure $\varsigma(dx,dy) = \mu_{0}(dx)\tilde{\mathbb{Q}}(dy)$.
Furthermore, it then follows that $\varsigma \ll \varsigma_{0}$ with
\begin{align}\label{3.2shizi}
\frac{d\varsigma}{d\varsigma_{0}}(x,y) = \frac{1}{Z(y)}\exp(-\Phi(x;y)).
\end{align}
Then we have the following theorem with a similar spirt of Theorem 2.5 in \cite{Bayesian2}.
\begin{theorem}\label{3.2cunzaiTheorem}
Assume that
\begin{enumerate}
  \item $\text{Im}(\mathcal{C}_{\eta}^{1/2}) = \text{Im}(\mathcal{C}_{\nu | x}^{1/2})$ and denote $E := \text{Im}(\mathcal{C}_{\eta}^{1/2})$;
  \item $\mathcal{G}(x) + \bar{\nu}_{x} \in E$;
  \item the operator $T := (\mathcal{C}_{\eta}^{-1/2}\mathcal{C}_{\nu |x}^{1/2})(\mathcal{C}_{\eta}^{-1/2}\mathcal{C}_{\nu |x}^{1/2})^{*} - I$
  is Hilbert-Schmidt in $\bar{E}$.
\end{enumerate}
In addition, assume that $\Phi : X\times Y \rightarrow \mathbb{R}$ is $\varsigma_{0}$ measurable, Assumption 2 holds and that, for $y$
$\mathbb{Q}_{0}$-a.s.,
\begin{align}\label{3.2tiaojian1}
Z(y) := \int_{X}\exp(-\Phi(x;y))\mu_{0}(dx) > 0.
\end{align}
Then the conditional distribution of $x|y$ exists under $\varsigma$, and is denoted by $\mu^{y}$.
Furthermore $\mu^{y} \ll \mu_{0}$ and, for $y$ $\varsigma$-a.s.,
\begin{align}\label{3.2jielun1}
\frac{d\mu^{y}}{d\mu_{0}}(x) = \frac{1}{Z(y)}\exp(-\Phi(x;y)).
\end{align}
Moreover, the measure $\mu^{y}$ is Lipschitz in the data $y$, with respect to the Hellinger distance:
if $\mu^{y}$ and $\mu^{y'}$ are two measures given by (\ref{3.2jielun1}) with data $y$ and $y'$ then there is
$C = C(r) > 0$ such that, for all $y$, $y'$ with $\max\{ |y|,|y'| \} \leq r$,
\begin{align}\label{LipschitzNew1}
d_{\text{Hell}}(\mu^{y},\mu^{y'}) \leq C |y-y'|.
\end{align}
Consequently all polynomially bounded functions of $x \in X$ are continuous in $y$.
In particular the mean and covariance operator are continuous in $y$.
\end{theorem}
\begin{remark}
Let $\nu$ be a common reference measure of measures $\mu$ and $\mu'$,
the Hellinger distance used in Theorem \ref{3.2cunzaiTheorem} is defined by
\begin{align*}
d_{\text{Hell}}(\mu,\mu') = \sqrt{\left( \frac{1}{2}\int\left( \sqrt{\frac{d\mu}{d\nu}} - \sqrt{\frac{d\mu'}{d\nu}} \right)^{2} d\nu \right)}.
\end{align*}
\end{remark}
\begin{proof}
Because the measure $\mathbb{Q}_{0}$ and the measure $\tilde{\mathbb{Q}}$ both are Gaussian measure, from
Feldman-Hajek theorem \cite{StochasticEquation1}, we can conclude that under conditions (1) to (3), $\tilde{\mathbb{Q}} \ll \mathbb{Q}_{0}$.
From $\tilde{\mathbb{Q}} \ll \mathbb{Q}_{0}$, we notice that $\Phi$ exists.
Note that the positive of $Z$ holds for $y$ $\varsigma_{0}$-almost surely, and hence by absolute continuity
of $\varsigma$ with respect to $\varsigma_{0}$, for $y$ $\varsigma$-almost surely.
Now by Theorem 6.29 in \cite{Bayesian3}, the first result follows.
For the Lipschitz continuity, it could be proved by the method used in the proof of Theorem 4.4 in \cite{Bayesian1}.
The reason is that we impose similar conditions on the potential $\Phi$ and the properties of $\Phi$ is the key point of the proof.
So we omit the details here.
\end{proof}

In the sequel, we consider a simple case that is the operator $\mathcal{C}_{\eta}$ and $\mathcal{C}_{\nu | x}$ commute
with each other. Because $\mathcal{C}_{\eta}$ and $\mathcal{C}_{\nu | x}$ commute, there exists a complete orthonormal
system $(e_{k})$ in $Y$, and sequences $(\lambda_{k})$, $(r_{k})$ of positive numbers such that
\begin{align}\label{3.2zhegnjiao1}
\mathcal{C}_{\eta}e_{k} = \lambda_{k} e_{k}, \quad
\mathcal{C}_{\nu | x}e_{k} = r_{k} e_{k}, \quad k\in \mathbb{N}.
\end{align}
In order to provide a clear verification, we denote
\begin{align}\label{defH}
h = \mathcal{G}_{a}(x) + \bar{\nu}_{x}.
\end{align}
Hence, we have
$\tilde{\mathbb{Q}} = \mathcal{N}(h, \mathcal{C}_{\nu | x})$ and $h = \sum_{k = 1}^{\infty}h_{k}e_{k}$ with
$h_{k} = (h,e_{k})$.
\begin{lemma}\label{3.2jiaohuan1}
Let $\mathcal{C}_{\eta}, \mathcal{C}_{\nu | x} \in L_{1}^{+}(Y)$ be such that
$[\mathcal{C}_{\eta}, \mathcal{C}_{\nu |x}] := \mathcal{C}_{\eta}\mathcal{C}_{\nu |x} - \mathcal{C}_{\nu |x}\mathcal{C}_{\eta} = 0$.
Then $\tilde{\mathbb{Q}}$ and $\mathbb{Q}_{0}$ are equivalent if and only if
\begin{align}\label{3.2dengjialitaojian1}
\sum_{k = 1}^{\infty}\frac{h_{k}^{2}}{\lambda_{k} + r_{k}} < \infty, \quad \text{and} \quad
\sum_{k = 1}^{\infty}\frac{(\lambda_{k}-r_{k})^{2}}{(\lambda_{k}+r_{k})^{2}} < \infty.
\end{align}
\end{lemma}
\begin{proof}
The proof of this lemma inspired by the methods used in the proof of Theorem 2.9 in \cite{StochasticEquation2}.
Here, we only provide the different parts.
We may write
\begin{align*}
\mathcal{N}(0,\mathcal{C}_{\eta}) = \cheng_{k=1}^{\infty} \mathcal{N}(0,\lambda_{k}), \quad
\mathcal{N}(h,\mathcal{C}_{\nu | x}) = \cheng_{k=1}^{\infty} \mathcal{N}(h_{k}, r_{k}).
\end{align*}
Because
\begin{align*}
\frac{d\mathcal{N}(h_{k},r_{k})}{d\mathcal{N}(0,\lambda_{k})}(x_{k}) = \sqrt{\frac{\lambda_{k}}{r_{k}}}
\exp\left(-\frac{x_{k}^{2}}{2}\left( \frac{\lambda_{k} - r_{k}}{r_{k}\lambda_{k}} \right) + \frac{h_{k}x_{k}}{r_{k}}
- \frac{h_{k}^{2}}{2r_{k}}\right), \quad x_{k}\in \mathbb{R}, \,  k\in \mathbb{N},
\end{align*}
then through a tedious computation, we could obtain that
\begin{align*}
H(\tilde{\mathbb{Q}},\mathbb{Q}_{0})^{4} = \exp\left(-\sum_{k = 1}^{\infty}\frac{h_{k}^{2}}{\lambda_{k}+r_{k}}\right)
\prod_{k = 1}^{\infty} \left( 1-\frac{(\lambda_{k} - r_{k})^{2}}{(\lambda_{k} + r_{k})^{2}} \right).
\end{align*}
Hence, $H(\tilde{\mathbb{Q}},\mathbb{Q}_{0}) > 0$ if and only if
\begin{align*}
\sum_{k = 1}^{\infty}\frac{h_{k}^{2}}{\lambda_{k} + r_{k}} < \infty, \quad \text{and} \quad
\sum_{k = 1}^{\infty}\frac{(\lambda_{k}-r_{k})^{2}}{(\lambda_{k}+r_{k})^{2}} < \infty.
\end{align*}
Now the conclusion follows from the Kakutani theorem \cite{StochasticEquation2}.
\end{proof}

Based on the above lemma, Theorem \ref{3.2cunzaiTheorem} can be modified as follows.
\begin{theorem}\label{3.2cunzaiTheoremTwo}
Assume that $\mathcal{C}_{\eta}, \mathcal{C}_{\nu | x} \in L_{1}^{+}(Y)$ be such that
$[\mathcal{C}_{\eta}, \mathcal{C}_{\nu |x}] := \mathcal{C}_{\eta}\mathcal{C}_{\nu |x} - \mathcal{C}_{\nu |x}\mathcal{C}_{\eta} = 0$.
There exist a complete orthonormal
system $(e_{k})$ in $Y$, and sequences $(\lambda_{k})$, $(r_{k})$ of positive numbers such that
\begin{align*}
\mathcal{C}_{\eta}e_{k} = \lambda_{k} e_{k}, \quad
\mathcal{C}_{\nu | x}e_{k} = r_{k} e_{k}, \quad k\in \mathbb{N},
\end{align*}
and
\begin{align*}
\sum_{k = 1}^{\infty}\frac{h_{k}^{2}}{\lambda_{k} + r_{k}} < \infty, \quad  \quad
\sum_{k = 1}^{\infty}\frac{(\lambda_{k}-r_{k})^{2}}{(\lambda_{k}+r_{k})^{2}} < \infty,
\end{align*}
where $h, h_{k}$ are defined as in (\ref{defH}).
In addition, assume that $\Phi : X\times Y \rightarrow \mathbb{R}$ is $\varsigma_{0}$ measurable, Assumption 2 holds and that, for $y$
$\mathbb{Q}_{0}$-a.s.,
\begin{align}\label{3.2tiaojian2}
Z(y) := \int_{X}\exp(-\Phi(x;y))\mu_{0}(dx) > 0.
\end{align}
Then the conditional distribution of $x|y$ exists under $\varsigma$, and is denoted by $\mu^{y}$.
Furthermore $\mu^{y} \ll \mu_{0}$ and, for $y$ $\varsigma$-a.s.,
\begin{align}\label{3.2jielun2}
\frac{d\mu^{y}}{d\mu_{0}}(x) = \frac{1}{Z(y)}\exp(-\Phi(x;y)).
\end{align}
Moreover, the measure $\mu^{y}$ is Lipschitz in the data $y$, with respect to the Hellinger distance.
\end{theorem}

\begin{remark}
By some small modifications as stated in \cite{Bayesian1},
the above mentioned Theorem \ref{3.2cunzaiTheorem} and Theorem \ref{3.2cunzaiTheoremTwo} still hold when
$H$ is a separable Hilbert space.
\end{remark}

In the last part of this section, we provide an explanation for the relations of the Bayesian methods and the regularization methods.
For this, the MAP estimators and the Onsager-Machlup functional play an important role,
which can be seen from the work \cite{Dunlop2016MAP,MAP_detail,Helin2015Maximum}.
As in \cite{MAP_detail}, we define a function $I : X \rightarrow \mathbb{R}$ by
\begin{align}\label{OMFunctional}
I(x) = \left\{\begin{aligned}
& \Phi(x;y) + \frac{1}{2}\|x-\bar{x}\|_{E}^{2} \quad \text{if }x-\bar{x}\in E, \text{ and} \\
& + \infty \quad\quad\quad\quad\quad\quad\quad\quad \text{else.}
\end{aligned}
\right.
\end{align}
Here, $E$ denotes the Cameron-Martin space of the Gaussian measure $\mu_{0}$ on $X$.
The MAP estimate of a measure $\mu$ can be defined as follows.
\begin{definition}\label{defMAPestimate}
Let
\begin{align*}
M^{\epsilon} = \sup_{x\in X}\mu(B_{\epsilon}(x)).
\end{align*}
Any point $\hat{x}\in X$ satisfying
\begin{align*}
\lim_{\epsilon \rightarrow 0}\frac{\mu(B_{\epsilon}(\hat{x}))}{M^{\epsilon}} = 1
\end{align*}
is a MAP estimate for the measure $\mu$.
\end{definition}
With these definitions, we can show the following theorems.
\begin{theorem}\label{MAP1theorem}
Suppose that Assumption 2 hold. Assume also that there exists an $M\in\mathbb{R}$ such that $\Phi(x;y) \geq M$ for any $x\in X$.
\begin{itemize}
  \item Let $z^{\delta} = \argmax_{z\in X}\mu^{y}(B_{\delta}(z))$. There is a $\bar{z}\in E$ and a subsequence of $\{z^{\delta}\}_{\delta >0}$
  which converges to $\bar{z}$ strongly in $X$.
  \item The limit $\bar{z}$ is a MAP estimator and a minimizer of $I$.
\end{itemize}
\end{theorem}
\begin{corollary}\label{MAP2theorem}
Under the conditions of Theorem \ref{MAP1theorem}, we have the following.
\begin{itemize}
  \item Any MAP estimator, given by Definition \ref{defMAPestimate}, minimizes the Onsager-Machlup functional $I$.
  \item Any $z^{*}\in E$ which minimizes the Onsager-Machlup functional $I$ is a MAP estimator for measure $\mu^{y}$
  appeared in Theorem \ref{3.2cunzaiTheorem} or Theorem \ref{3.2cunzaiTheoremTwo}.
\end{itemize}
\end{corollary}
Considering Assumption 2, the proofs of Theorem \ref{MAP1theorem} and Corollary \ref{MAP2theorem} are exactly the same as in \cite{MAP_detail}.
So we only provide these two results.

%Finally, we should point out that our theory provides an infinite-dimensional Bayesian model error approach.
%The assumption of $\epsilon$ seems unreasonable; however, in the finite-dimensional case,
%the Gaussian measure is the only choice.
%Hence, at present, it may be unreachable to develop a theory for a non-Gaussian measure in
%infinite-dimensional space. Developing a theory with more appropriate assumptions of the model reduction error is an
%interesting problem, which will be the focus of our future work.

\subsection{Small error limits}

This section is devoted to a small error limit problem, which could be seen as a result of posterior consistency:
the idea that the posterior concentrates near the truth that give rise to the data in the small error limits.
The studies here are inspired by the work \cite{Bissantz2004Consistency,MAP_detail}.
For notational simplicity, we assume $\bar{x} = 0$ in this section.
We assume $\mathcal{G}$ be the forward operator without model reduction error, $\mathcal{G}_{n}$ be the
forward operator with model reduction error $\frac{1}{n}\epsilon_{n}$ with $n\in \mathbb{N}$, where
$\epsilon_{n} \sim \mathcal{N}(\bar{\epsilon},\mathcal{C}_{\epsilon})$ defined similarly as in Assumption 2.
In the following, we denote $x^{\dag}$ to be the truth.
Still considering $X$ be a separable Hilbert space and $Y = \mathbb{R}^{J}$, the problem can be written as follow
\begin{align}\label{appromodel1}
y_{n} = \mathcal{G}_{n}(x^{\dag}) + \frac{1}{n}\epsilon_{n} + \frac{1}{n}\eta_{n}
\end{align}
for $n\in\mathbb{N}$ and $\eta_{n} \sim \mathbb{Q}_{0} = \mathcal{N}(0,\mathcal{C}_{\eta})$ defined similarly as in Assumption 2.
Similar to (\ref{3.2nota1}) and (\ref{3.2nota2}), we can define $\nu_{n}$, $\nu_{n}|x$. Then we have
$\nu_{n} | x \sim \mathcal{N}(\bar{\nu}_{x}^{n},\mathcal{C}_{\nu | x}^{n})$ where
\begin{align}\label{Ngongshi}
\begin{split}
& \quad \bar{\nu}_{x}^{n} = \frac{1}{n}\left(\bar{\epsilon} + \mathcal{C}_{\epsilon x}\mathcal{C}_{x}^{-1}(x-\bar{x})\right)
= \frac{1}{n}\bar{\nu}_{x},   \\
& \mathcal{C}_{\nu | x}^{n} = \frac{1}{n^{2}}\left(\mathcal{C}_{\eta} + \mathcal{C}_{\epsilon}
- \mathcal{C}_{\epsilon x}\mathcal{C}_{x}^{-1}\mathcal{C}_{x\epsilon}\right) = \frac{1}{n^{2}}\mathcal{C}_{\nu | x}.
\end{split}
\end{align}

Assume $\mu_{0}$ satisfy Assumption 2, we have the following formula for the posterior measure:
\begin{align}\label{poseriorMeasure}
\frac{d\mu^{y_{n}}}{d\mu_{0}}(x) \propto \exp \left( -\frac{n^{2}}{2}|y_{n} - \mathcal{G}(x) - \bar{\nu}_{x}^{n}|_{\mathcal{C}_{\nu|x}}^{2} \right)
\end{align}
If we assume $\mathcal{G}, \mathcal{G}_{n}$ are uniformly Lipschitz continuous on bounded sets, by Theorem \ref{MAP1theorem}
and Corollary \ref{MAP2theorem}, the MAP estimate of the above measure are the minimizers of
\begin{align}\label{MAPFormula1}
I_{n}(x) := \|x\|_{E}^{2} + n^{2} |y_{n}-G_{n}(x) - \bar{\nu}_{x}^{n}|_{\mathcal{C}_{\nu|x}}^{2},
\end{align}
where $E$ denotes the Cameron-Martin space of the Gaussian measure $\mu_{0}$ on $X$ as in the previous section.
With there preparations, we can show the main result of this section as follows.
\begin{theorem}\label{uniformCon}
Assume that $\mathcal{G}_{n},\mathcal{G} : X \rightarrow \mathbb{R}^{J}$ are uniformly Lipschitz on bounded sets and
$x^{\dag} \in E$. For every $x\in E$, we assume
\begin{align}\label{convergeAss}
\lim_{n\rightarrow +\infty}|\mathcal{G}_{n}(x) - \mathcal{G}(x)| = 0.
\end{align}
For every $n \in \mathbb{N}$, let $x_{n}\in E$ be a minimizer of $I_{n}$ given by (\ref{MAPFormula1}).
Then there exist a $x^{*} \in E$ and a subsequence of $\{x_{n}\}_{n\in\mathbb{N}}$ that converges weakly to $x^{*}$ in $E$,
almost surely. For any such $x^{*}$, we have $\mathcal{G}(u^{*}) = \mathcal{G}(u^{\dag})$.
\end{theorem}
%\begin{remark}\label{helixingforjiashe}
%Intuitively, we have $\frac{1}{n}\epsilon_{n} = \mathcal{G}(x) - \mathcal{G}_{n}(x).$
%Since $|\epsilon_{n}|$ are bounded almost surely, we have
%$\frac{1}{n}|\epsilon_{n}| \rightarrow 0 \quad \text{as }n\rightarrow \infty$, almost surely.
%So the assumption (\ref{convergeAss}) in the above theorem is appropriate.
%\end{remark}
\begin{proof}
For two column vectors $a,b \in \mathbb{R}^{J}$, denote
$\langle a,b \rangle_{\mathcal{C}_{\nu|x}} := a^{T}\cdot\mathcal{C}_{\nu|x}^{-1}\cdot b$, where $a^{T}$ represents the transpose of $a$.
Notice (\ref{appromodel1}) and (\ref{Ngongshi}), we obtain
\begin{align*}
I_{n} = & \|x\|_{E}^{2} + n^{2} \left|\mathcal{G}_{n}(x^{\dag})-\mathcal{G}_{n}(x) + \frac{1}{n}\epsilon_{n}-\frac{1}{n}\bar{\epsilon}
+ \frac{1}{n}\eta_{n} - \frac{1}{n}\mathcal{C}_{\epsilon x}\mathcal{C}_{x}^{-1}x\right|_{\mathcal{C}_{\nu|x}}^{2}   \\
= & \|x\|_{E}^{2} + n^{2}|\mathcal{G}_{n}(x^{\dag}) - \mathcal{G}_{n}(x)|_{\mathcal{C}_{\nu|x}}^{2}
+ |\epsilon_{n}-\bar{\epsilon}+\eta_{n}-\mathcal{C}_{\epsilon x}\mathcal{C}_{x}^{-1}x|_{\mathcal{C}_{\nu|x}}^{2} \\
& + 2n \langle\mathcal{G}_{n}(x^{\dag})-\mathcal{G}_{n}(x),
\epsilon_{n}-\bar{\epsilon}+\eta_{n}-\mathcal{C}_{\epsilon x}\mathcal{C}_{x}^{-1}x\rangle_{\mathcal{C}_{\nu|x}}
\end{align*}
Define
\begin{align*}
J_{n} = & \frac{1}{n^{2}}\|x\|_{E}^{2} + |\mathcal{G}_{n}(x^{\dag}) - \mathcal{G}_{n}(x)|_{\mathcal{C}_{\nu|x}}^{2}
+ \frac{1}{n^{2}}|\epsilon_{n}-\bar{\epsilon}+\eta_{n}-\mathcal{C}_{\epsilon x}\mathcal{C}_{x}^{-1}x|_{\mathcal{C}_{\nu|x}}^{2} \\
& + \frac{2}{n}\langle\mathcal{G}_{n}(x^{\dag})-\mathcal{G}_{n}(x),
\epsilon_{n}-\bar{\epsilon}+\eta_{n}-\mathcal{C}_{\epsilon x}\mathcal{C}_{x}^{-1}x\rangle_{\mathcal{C}_{\nu|x}}.
\end{align*}
We have
\begin{align*}
\argmin_{x}I_{n} = \argmin_{x}J_{n}.
\end{align*}
Define $x_{n}\in E$ as follow
\begin{align*}
x_{n} = \argmin_{x\in E}J_{n}(x).
\end{align*}
The existence of $x_{n}$ obviously follows Theorem 5.4 in \cite{Bayesian3}.  By the definition of $x_{n}$, we have
\begin{align*}
\frac{1}{n^{2}}\|x_{n}\|_{E}^{2} & + |\mathcal{G}_{n}(x^{\dag})-\mathcal{G}_{n}(x_{n})|_{\mathcal{C}_{\nu|x}}^{2}
+ \frac{1}{n^{2}}|\epsilon_{n}-\bar{\epsilon}+\eta_{n}-\mathcal{C}_{\epsilon x}\mathcal{C}_{x}^{-1}x_{n}|_{\mathcal{C}_{\nu|x}}^{2} \\
& + \frac{2}{n}\langle\mathcal{G}_{n}(x^{\dag})-\mathcal{G}_{n}(x_{n}),
\epsilon_{n}-\bar{\epsilon}+\eta_{n}-\mathcal{C}_{\epsilon x}\mathcal{C}_{x}^{-1}x_{n}\rangle_{\mathcal{C}_{\nu|x}} \\
& \leq \frac{1}{n^{2}}\|x^{\dag}\|_{E}^{2}
+ \frac{1}{n^{2}}|\epsilon_{n}-\bar{\epsilon}+\eta_{n}-\mathcal{C}_{\epsilon x}\mathcal{C}_{x}^{-1}x^{\dag}|_{\mathcal{C}_{\nu|x}}^{2}.
\end{align*}
Simple calculations yields
\begin{align}\label{bijiantui1}
\begin{split}
\frac{1}{n^{2}}\|x_{n}\|_{E}^{2} + & |\mathcal{G}_{n}(x^{\dag})-\mathcal{G}_{n}(x_{n})|_{\mathcal{C}_{\nu|x}}^{2}
+ \frac{1}{n^{2}}|\epsilon_{n}-\bar{\epsilon}+\eta_{n}-\mathcal{C}_{\epsilon x}\mathcal{C}_{x}^{-1}x_{n}|_{\mathcal{C}_{\nu|x}}^{2} \\
\leq & \frac{1}{n^{2}}\|x^{\dag}\|_{E}^{2}
+ \frac{1}{n^{2}}|\epsilon_{n}-\bar{\epsilon}+\eta_{n}-\mathcal{C}_{\epsilon x}\mathcal{C}_{x}^{-1}x^{\dag}|_{\mathcal{C}_{\nu|x}}^{2} \\
& + \frac{2}{n}|\mathcal{G}_{n}(x^{\dag})-\mathcal{G}_{n}(x_{n})|_{\mathcal{C}_{\nu|x}}
|\epsilon_{n}-\bar{\epsilon}+\eta_{n}-\mathcal{C}_{\epsilon x}\mathcal{C}_{x}^{-1}x_{n}|_{\mathcal{C}_{\nu|x}}.
\end{split}
\end{align}
Using Young's inequality, we have
\begin{align}\label{bijiantui2}
\begin{split}
& \frac{2}{n}|\mathcal{G}_{n}(x^{\dag})-\mathcal{G}_{n}(x_{n})|_{\mathcal{C}_{\nu|x}}
|\epsilon_{n}-\bar{\epsilon}+\eta_{n}-\mathcal{C}_{\epsilon x}\mathcal{C}_{x}^{-1}x_{n}|_{\mathcal{C}_{\nu|x}} \\
\leq & \frac{m-1}{m}|\mathcal{G}_{n}(x^{\dag})-\mathcal{G}_{n}(x_{n})|_{\mathcal{C}_{\nu|x}}^{2}
+ \frac{m}{m-1}\frac{1}{n^{2}}|\epsilon_{n}-\bar{\epsilon}+\eta_{n}-\mathcal{C}_{\epsilon x}\mathcal{C}_{x}^{-1}x_{n}|_{\mathcal{C}_{\nu|x}}^{2}
\end{split}
\end{align}
for a large enough real number $m$ which will be specified later.
Substituting (\ref{bijiantui2}) into (\ref{bijiantui1}), we have
\begin{align}\label{bijiantui3}
\begin{split}
\frac{1}{n^{2}}\|x_{n}\|_{E}^{2} + & \frac{1}{m}|\mathcal{G}_{n}(x^{\dag})-\mathcal{G}_{n}(x_{n})|_{\mathcal{C}_{\nu|x}}^{2}
\leq \frac{1}{n^{2}}|\epsilon_{n}-\bar{\epsilon}+\eta_{n}-\mathcal{C}_{\epsilon x}\mathcal{C}_{x}^{-1}x^{\dag}|_{\mathcal{C}_{\nu|x}}^{2} \\
& + \frac{1}{n^{2}}\|x^{\dag}\|_{E}^{2}
+ \frac{1}{(m-1)n^{2}}|\epsilon_{n}-\bar{\epsilon}+\eta_{n}
-\mathcal{C}_{\epsilon x}\mathcal{C}_{x}^{-1}x_{n}|_{\mathcal{C}_{\nu|x}}^{2}.
\end{split}
\end{align}
Now we concentrate on the third term on the right-hand side of the above inequality.
By simple calculations, we have
\begin{align}\label{bijiantui4}
\begin{split}
& \frac{1}{(m-1)n^{2}}|\epsilon_{n}-\bar{\epsilon}+\eta_{n}-\mathcal{C}_{\epsilon x}\mathcal{C}_{x}^{-1}x_{n}|_{\mathcal{C}_{\nu|x}}^{2} \\
\leq & \frac{2}{m-1}\frac{1}{n^{2}}\left( |\epsilon_{n}-\bar{\epsilon}+\eta_{n}|_{\mathcal{C}_{\nu|x}}^{2} +
|\mathcal{C}_{\epsilon x}\mathcal{C}_{x}^{-1}x_{n}|_{\mathcal{C}_{\nu|x}}^{2} \right)   \\
\leq & \frac{2}{m-1}\frac{1}{n^{2}}|\epsilon_{n}-\bar{\epsilon}+\eta_{n}|_{\mathcal{C}_{\nu|x}}^{2} +
\frac{2}{m-1}\frac{1}{n^{2}}C_{1}\|x_{n}\|_{E}^{2}.
\end{split}
\end{align}
Here, we take $m$ large enough such that
$$\frac{2}{m-1}C_{1} \leq \frac{1}{2}.$$
Then substituting (\ref{bijiantui4}) into (\ref{bijiantui3}), we obtain
\begin{align}\label{bijiantui5}
\begin{split}
\frac{1}{2n^{2}}\|x_{n}\|_{E}^{2} & + \frac{1}{m}|\mathcal{G}_{n}(x^{\dag})-\mathcal{G}_{n}(x_{n})|_{\mathcal{C}_{\nu|x}}^{2}
\leq \frac{2}{(m-1)n^{2}}|\epsilon_{n}-\bar{\epsilon}+\eta_{n}|_{\mathcal{C}_{\nu|x}}^{2}   \\
& + \frac{1}{n^{2}}\|x^{\dag}\|_{E}^{2}
+ \frac{1}{n^{2}}|\epsilon_{n}-\bar{\epsilon}+\eta_{n}-\mathcal{C}_{\epsilon x}\mathcal{C}_{x}^{-1}x^{\dag}|_{\mathcal{C}_{\nu|x}}^{2}.
\end{split}
\end{align}
Taking expectation on both sides of the above inequality, we obtain
\begin{align}\label{bijiantui6}
\begin{split}
\frac{1}{2n^{2}}\mathbb{E}\|x_{n}\|_{E}^{2} + \frac{1}{m}\mathbb{E}|\mathcal{G}_{n}(x^{\dag})-\mathcal{G}_{n}(x_{n})|_{\mathcal{C}_{\nu|x}}^{2}
\leq \frac{1}{n^{2}}\left( \|x^{\dag}\|_{E}^{2} + \frac{2}{m-1}K_{1} + K_{2} \right),
\end{split}
\end{align}
where
\begin{align*}
K_{1} & := \mathbb{E}|\epsilon_{n}-\bar{\epsilon}+\eta_{n}|_{\mathcal{C}_{\nu|x}}^{2}, \\
K_{2} & := \mathbb{E}|\epsilon_{n}-\bar{\epsilon}+\eta_{n}-\mathcal{C}_{\epsilon x}\mathcal{C}_{x}^{-1}x^{\dag}|_{\mathcal{C}_{\nu|x}}^{2}.
\end{align*}
Obviously, $K_{1}$ and $K_{2}$ are bounded and independent of $n$. Hence, (\ref{bijiantui6}) implies that
\begin{align}\label{bijiantui7}
\mathbb{E}|\mathcal{G}_{n}(x^{\dag})-\mathcal{G}_{n}(x_{n})|_{\mathcal{C}_{\nu|x}}^{2} \rightarrow 0 \quad \text{as }n\rightarrow \infty
\end{align}
and
\begin{align}\label{bijiantui8}
\mathbb{E}\|x_{n}\|_{E}^{2} \leq 2\|x^{\dag}\|_{E}^{2} + \frac{4}{m-1}K_{1} + 2K_{2}.
\end{align}
Similar to the proof of (4.4) in \cite{MAP_detail}, from (\ref{bijiantui8}), we obtain that there exist $x^{*}\in E$ and
a subsequence $\{x_{n_{k}(k)}\}_{k\in\mathbb{N}}$ of $\{x_{n}\}_{x\in\mathbb{N}}$ such that
\begin{align}\label{bijiantui9}
\mathbb{E}\langle x_{n_{k}(k)}, v \rangle_{E} \rightarrow \mathbb{E}\langle x^{*}, v \rangle_{E} \quad \text{for any }v\in E.
\end{align}
By (\ref{bijiantui7}), we have $|\mathcal{G}_{n_{k}(k)}(x^{\dag})-G_{n_{k}(k)}(x_{n_{k}(k)})|_{\mathcal{C}_{\nu|x}} \rightarrow 0$
in probability as $k\rightarrow \infty$.
Therefore, there exists a subsequence $\{x_{m(k)}\}$ of $\{x_{n_{k}(k)}\}$ such that
\begin{align*}
\mathcal{G}_{m(k)}(x^{\dag}) - \mathcal{G}_{m(k)}(x_{m(k)}) \rightarrow 0 \quad \text{a.s. as }k \rightarrow \infty.
\end{align*}
By our hypothesis, we know that $\mathcal{G}_{m(k)}(x^{\dag}) \rightarrow \mathcal{G}(x^{\dag})$ as $k\rightarrow\infty$.
Hence, we have
\begin{align*}
\mathcal{G}_{m(k)}(x_{m(k)}) \rightarrow \mathcal{G}(x^{\dag}) \quad \text{a.s. as }k \rightarrow \infty.
\end{align*}
From (\ref{bijiantui9}), we obtain $\langle x_{m(k)}-x^{*}, v \rangle_{E} \rightarrow 0$ in probability as $k\rightarrow\infty$,
and so there exists a subsequence $\{x_{\hat{m}(k)}\}$ of $\{x_{m(k)}\}$ such that $x_{\hat{m}(k)}$ converges weakly to $x^{*}$
in $E$ almost surely as $k \rightarrow \infty$. Because $E$ is compactly embedded in $X$, this implies that
$x_{\hat{m}(k)}\rightarrow x^{*}$ in $X$ almost surely as $k\rightarrow\infty$.
Since
\begin{align*}
|\mathcal{G}(x^{*}) - \mathcal{G}_{\hat{m}(k)}(x_{\hat{m}(k)})|
\leq |\mathcal{G}(x^{*}) - \mathcal{G}_{\hat{m}(k)}(x^{*})|
+ |\mathcal{G}_{\hat{m}(k)}(x^{*}) - \mathcal{G}_{\hat{m}(k)}(x_{\hat{m}(k)})|,
\end{align*}
by hypothesis (\ref{convergeAss}) and $\mathcal{G}_{n}$ are uniformly Lipschitz bounded, we obtain
\begin{align*}
\mathcal{G}_{\hat{m}(k)}(x_{\hat{m}(k)}) \rightarrow \mathcal{G}(x^{*}) \quad \text{a.s. as }k \rightarrow \infty.
\end{align*}
Thus, the proof has been finished.
\end{proof}

In the above theorem, we assumed the truth $x^{\dag}$ belongs to the Cameron-Martin space $E$.
We could show a weaker convergence result when $x^{\dag}$ just belongs to $X$.

\begin{theorem}\label{weakConvergence}
Suppose that $\mathcal{G}_{n}$, $\mathcal{G}$ and $x_{n}$ satisfy the assumptions of Theorem \ref{uniformCon},
and that $x^{\dag}\in X$. Then there exists a subsequence of $\{\mathcal{G}_{n}(x_{n})\}_{n\in\mathbb{N}}$ converging
to $\mathcal{G}(x^{\dag})$ almost surely.
\end{theorem}

We show the proof of Theorem \ref{uniformCon} in detail. Combining the calculations above with the
arguments used in the proof of Corollary 4.3 in \cite{MAP_detail}, it is not difficult to write down the complete
proof of Theorem \ref{weakConvergence}. For concise, we omit the details here.

\subsection{Apply to an inverse scattering problem}

Before going further, we provide a hypothesis on the covariance operator.

\textbf{Assumption 4}: The operator $A$, densely defined on the Hilbert space $\mathcal{H} = L^{2}(B_{R};\mathbb{R}^{n})$, satisfies
the following properties:
\begin{enumerate}
  \item $A$ is positive-definite, self-adjoint and invertible;
  \item the eigenfunctions $\{ \phi_{j} \}_{j \in \mathbb{N}}$ of $A$, form an orthonormal basis for $\mathcal{H}$;
  \item there are $C^{\pm} > 0$ such that the eigenvalues satisfy $\alpha_{j} \approx j^{2/n}$, for all $j \in \mathbb{N}$;
  \item there is $C > 0$ such that
  \begin{align*}
  \sup_{j\in\mathbb{N}}\left( \|\phi_{j}\|_{L^{\infty}} + \frac{1}{j^{1/n}}\text{Lip}(\phi_{j}) \right) \leq C.
  \end{align*}
\end{enumerate}

\subsubsection{Without model reduction error}\label{3.2yingyong}

As a warm up, let us consider the case without model reduction error, which can be covered by the theory developed in \cite{Bayesian2}.
Let $B_{R} \subset \mathbb{R}^{2}$ be the ball mentioned in Section \ref{2.1lossmain}.
We set $X = C_{u}(B_{R})$, define $V := H^{1}(B_{R})$. Let $\ell_{j}$ with $j = 1,2,\cdots,J$ are linear functionals on $V$,
that means $\ell_{j} \in V^{*}$ where $V^{*}$ is the dual space of $V$.
Define
\begin{align}\label{3.1defineQ}
\tilde{q}(x) := \log(1+q(x)),
\end{align}
and as in Section \ref{2.1lossmain}, we denote $u^{s}(x) = \mathfrak{S}(e^{\tilde{q}} - 1)u^{\text{inc}}$.
According to Theorem \ref{2.1generalK}, we may know that $u^{s} \in H^{1}(B_{R})$.
Hence, in our setting, the unknown function $x$ should be function $\tilde{q}$ and the observation operator could be defined as follows:
\begin{align}\label{3.1forwardOpe}
\mathcal{G}(\tilde{q}) = \{\mathcal{G}_{j}(\tilde{q})\}_{j=1}^{J} := \{ \ell_{j}(\mathfrak{S}(e^{\tilde{q}} - 1)u^{\text{inc}}) \}_{j=1}^{J}.
\end{align}
We take a prior on $\tilde{q}$ to be the measure $\mathcal{N}(0,A^{-s})$ with $s > 1$ where $A$ is an operator satisfy Assumption 4 with $n = 1$.
From Theorem 2.18 in \cite{Bayesian1}, we obtain that $\mu_{0}(X) = 1$.

Denote $\mathbb{Q}_{0} = \mathcal{N}(0,\Gamma)$, $\mathbb{Q}_{\tilde{q}} = \mathcal{N}(\mathcal{G}(\tilde{q}),\Gamma)$.
Taking $B_{1}$ as a unit ball in $X$.
Since
\begin{align}\label{3.1yanzheng}
|\mathcal{G}(\tilde{q})| \leq \|\mathfrak{S}(e^{\tilde{q}}-1)u^{\text{inc}}\|_{H^{1}} \leq C < \infty,
\end{align}
and $\eta \sim \mathcal{N}(0,\Gamma)$, noting that $y$ is $\mathbb{Q}_{0}$-a.s. finite, we have for some $M = M(y) < \infty$
\begin{align*}
\sup_{\tilde{q}\in B_{1}} \frac{1}{2}|\Gamma^{-1/2}(y-\mathcal{G}(\tilde{q}))|^{2}  < M.
\end{align*}
Denote $Z = \int_{X}\exp(-\Phi(x;y)) \mu_{0}(dx)$, by Theorem 6.28 in \cite{Bayesian2}, we know that
\begin{align*}
Z \geq \int_{B_{1}} \exp(-M)\mu_{0}(d\tilde{q}) = \exp(-M)\mu_{0}(B_{1}) > 0.
\end{align*}
Thus, by Theorem 2.1 in \cite{Bayesian2}, we obtain
\begin{align*}
\frac{d\mu^{y}}{d\mu_{0}}(\tilde{q}) = \frac{1}{Z(y)}\exp(-\Phi(\tilde{q};y)),
\end{align*}
where
\begin{align}\label{3.1formofPhi}
\Phi(\tilde{q};y) = \frac{1}{2}|\Gamma^{-1/2}(y-\mathcal{G}(\tilde{q}))|^{2}.
\end{align}
Considering (\ref{3.1yanzheng}) and Theorem \ref{2.1LipCon}, we easily verified that $\Phi$ in (\ref{3.1formofPhi})
satisfies Assumption 3. Hence, we actually proved the following theorem.
\begin{theorem}\label{3.1mainTheorem}
For the two-dimensional inverse scattering problem related with the loss-dominated fractional Helmholtz equation
(problem (\ref{3.1jiaformula}) with $\mathcal{G}$ given by (\ref{3.1forwardOpe})), if we assume
$\tilde{q} = \log(1+q) \sim \mu_{0}$ where $\mu_{0} = \mathcal{N}(0,A^{-s})$ with $s>1$. In addition, we assume
$\eta \in \mathbb{R}^{J}$, $\eta \sim \mathbb{Q}_{0}$ where $\mathbb{Q}_{0} = \mathcal{N}(0,\Gamma)$.
Then the posterior measure $\mu^{y}$ exists and absolutely continuous with respect to $\mu_{0}$ with Randon-Nikodym derivative given by
\begin{align*}
\frac{d\mu^{y}}{d\mu_{0}}(\tilde{q}) = \frac{1}{Z(y)}\exp(-\Phi(\tilde{q};y)),
\end{align*}
where
\begin{align}\label{3.1formofPhi2}
\Phi(\tilde{q};y) = \frac{1}{2}|\Gamma^{-1/2}(y-\mathcal{G}(\tilde{q}))|^{2},\quad Z(y) = \int_{C_{u}(B_{R})}\exp(-\Phi(\tilde{q};y)) \mu_{0}(d\tilde{q}),
\end{align}
In addition, the measure $\mu^{y}$ is continuous in the Hellinger metric with respect to the data $y$.
\end{theorem}

Apart from this well-posedness result, the approximation results \cite{Bayesian6} and the MAP estimators results \cite{MAP_detail}
can be obtained under the aforementioned setting.

\subsubsection{With model reduction error}

For the fractional Helmholtz equation in some unbounded domain, we usually need to calculate it by adding some artificial boundary conditions
(e.g., absorbing boundary conditions or perfectly matched layer methods).
As a simple illustration, we will analyze absorbing boundary conditions with the following form:
\begin{align}\label{3.2boundary}
\partial_{\mathbf{n}}u^{s} = i k u^{s} \quad \text{on }\partial D,
\end{align}
where $D \subset \mathbb{R}^{n}$ is a bounded Lipschitz domain.
With this boundary condition, our problem becomes
\begin{align}\label{3.2boundedFraProFinal}
\left\{
\begin{aligned}
& \Delta u^{s} + i\omega\tau A_{L}u^{s} + k^{2}(1+q)u^{s} = (-k^{2}q-i\omega\tau k^{2\gamma+1})u^{\text{inc}} \quad \text{in }D   \\
& \partial_{\mathbf{n}}u^{s} = i k u^{s}, \quad \text{on }\partial D,
\end{aligned}
\right.
\end{align}
where $\supp(q)\subset D$.
As in Section \ref{2.1lossmain}, we denote $u^{s} = \mathfrak{S}_{a}(q)u^{\text{inc}}$.
The operator $\mathfrak{S}_{a}$ and the operator $\mathfrak{S}$ in Section \ref{2.1lossmain} will be similar if the domain $D$ is large enough.
For the operator $\mathfrak{S}_{a}$, Theorem \ref{2.1smallKexistence}, Theorem \ref{2.1generalK} and Theorem \ref{2.1LipCon}
can be established similarly (actually, the proof will be simpler).
Denote $\tilde{\epsilon} = \mathfrak{S}_{a}(q)u^{\text{inc}} - \mathfrak{S}(q)u^{\text{inc}}$, then $\tilde{\epsilon}$ means the
system reduction error brought by the absorbing boundary condition.

Similar to Subsection \ref{3.2yingyong}, define $V := H^{1}(D)$. Let $\ell_{j}$ with $j = 1,2,\cdots,J$ are linear functionals on $V$,
that means $\ell_{j} \in V^{*}$ where $V^{*}$ is the dual space of $V$.
Define
\begin{align}\label{3.2defineQ}
\tilde{q}(x) := \log(1+q(x)),
\end{align}
then the forward operator will be defined as follows:
\begin{align}\label{3.2forwardOpe}
\mathcal{G}_{a}(\tilde{q}) = \{\mathcal{G}_{a}^{j}(\tilde{q})\}_{j=1}^{J} := \{ \ell_{j}(\mathfrak{S}_{a}(e^{\tilde{q}} - 1)u^{\text{inc}}) \}_{j=1}^{J}.
\end{align}
The system reduction error can be defined as
\begin{align}\label{3.2systemError}
\epsilon = \{ \ell_{j}(\tilde{\epsilon}) \}_{j=1}^{J} =
\{ \ell_{j}({\mathfrak{S}_{a}(q)u^{\text{inc}} - \mathfrak{S}(q)u^{\text{inc}}}) \}_{j=1}^{J}.
\end{align}
Based on these considerations, our model can be presented as follows:
\begin{align}\label{3.2allpyModel}
y = \mathcal{G}_{a}(\tilde{q}) + \epsilon + \eta,
\end{align}
where $\epsilon, \eta \in \mathbb{R}^{J}$.
In our setting, the covariance operator $\mathcal{C}_{\eta}$ and $\mathcal{C}_{\nu | \tilde{q}}$
(in our setting we change $x$ to $\tilde{q}$) are symmetric matrix.
Hence, we could obtain the following form of potential $\Phi$
\begin{align}\label{3.2applyPotential}
\Phi(\tilde{q};y) = \frac{1}{2}|\mathcal{C}_{\nu | \tilde{q}}^{-1/2}(y - \mathcal{G}_{a}(\tilde{q}) - \bar{\nu}_{\tilde{q}})|^{2}.
\end{align}
Taking $s>0$ is a small positive number, $\mu_{0} = \mathcal{N}(0,A^{-2(s+1)})$ and $X = H^{1+s}(D)$,
then by Lemma 6.27 in \cite{Bayesian3} we conclude that $\mu_{0}(X) = 1$ if $A$ satisfies Assumption 4.

\begin{theorem}\label{3.2mainTheorem}
For the two-dimensional inverse scattering problem concerned with the loss-dominated fractional Helmholtz equation
with absorbing boundary condition (problem (\ref{3.2boundedFraProFinal}) with $\mathcal{G}_{a}$ given by (\ref{3.2forwardOpe})), if we assume
$\tilde{q} = \log(1+q) \sim \mu_{0}$ where $\mu_{0} = \mathcal{N}(0,A^{-2(s+1)})$ with $s>0$. In addition, we assume
$\eta \in \mathbb{R}^{J}$, $\eta \sim \mathbb{Q}_{0}$ where $\mathbb{Q}_{0} = \mathcal{N}(0,\mathcal{C}_{\eta})$,
$(\epsilon,\tilde{q}) \in \mathcal{H}:= \mathbb{R}^{J}\times H^{1+s}(D)$ distributed
according to a Gaussian measure $\mathcal{N}((\bar{\epsilon},0),\mathcal{C})$.
Denote $\nu$ and $\nu|\tilde{q}$ have same meaning with (\ref{3.2nota1}) and (\ref{3.2nota2}).
Then the posterior measure $\mu^{y}$ exists and absolutely continuous with respect to $\mu_{0}$ with Randon-Nikodym derivative given by
\begin{align*}
\frac{d\mu^{y}}{d\mu_{0}}(\tilde{q}) = \frac{1}{Z(y)}\exp(-\Phi(\tilde{q};y)),
\end{align*}
where
\begin{align}\label{3.2formofPhi2}
\Phi(\tilde{q};y) = \frac{1}{2}|\mathcal{C}_{\nu | \tilde{q}}^{-1/2}(y - \mathcal{G}_{a}(\tilde{q}) - \bar{\nu}_{\tilde{q}})|^{2},
\end{align}
and
\begin{align*}
Z(y) = \int_{H^{1+s}(D)}\exp(-\Phi(\tilde{q};y))\mu_{0}(d\tilde{q}).
\end{align*}
In addition, the measure $\mu^{y}$ is continuous in the Hellinger metric with respect to the data $y$.
\end{theorem}
\begin{proof}
To conclude the proof of this theorem, we need to check $Z > 0$ $\mathbb{Q}_{0}$-a.s. and
$\Phi$ defined in (\ref{3.2formofPhi2}) satisfy Assumption 3.
For the former one, notice that
\begin{align*}
|\mathcal{G}_{a}(\tilde{q})| \leq \|\mathfrak{S}_{a}(\tilde{q})\|_{H^{1}(D)} \leq C < \infty,
\end{align*}
where $C$ depend on $\|\tilde{q}\|_{L^{\infty}(D)}$ which could be bounded by $\|\tilde{q}\|_{H^{1+s}(D)}$.
Because $\eta \sim \mathcal{N}(0,\mathcal{C}_{\eta})$, notice that $y$ is $\mathbb{Q}_{0}$-a.s. finite, we have for some $M = M(y) < \infty$
\begin{align*}
\frac{1}{2}|\mathcal{C}_{\nu | \tilde{q}}^{-1/2}(y - \mathcal{G}_{a}(\tilde{q}) - \bar{\nu}_{\tilde{q}})|^{2} < M.
\end{align*}
Hence, by Theorem 6.28 in \cite{Bayesian2}, we know that
\begin{align*}
Z \geq \int_{B_{1}} \exp(-M)\mu_{0}(d\tilde{q}) = \exp(-M)\mu_{0}(B_{1}) > 0.
\end{align*}
To check $\Phi$ defined in (\ref{3.2formofPhi2}) satisfy Assumption 3, we should notice the following fact
\begin{align}\label{3.2lipschitz}
\|\mathfrak{S}_{a}(\tilde{q}_{1}) - \mathfrak{S}_{a}(\tilde{q}_{2})\|_{H^{1}(D)}
\leq C \|\tilde{q}_{1} - \tilde{q}_{2}\|_{L^{\infty}(D)} \leq C \|\tilde{q}_{1} - \tilde{q}_{2}\|_{H^{1+s}(D)},
\end{align}
which can be verified easily by employing similar methods used in the proof of Theorem \ref{2.1LipCon}.
Considering (\ref{3.2lipschitz}), Assumption 3 can be verified by simple calculations.
Hence, the proof is completed.
\end{proof}

\begin{remark}
We provide a simple example, which only incorporates model reduction error induced by the absorbing boundary
condition. Using a similar method, we may incorporate some other kinds of model reduction error (e.g.,
induced by perfectly matched layer).
\end{remark}

Under the above setting, we easily know that the Onsager-Machlup function has the following form:
\begin{align}\label{omfunctionExp}
I(\tilde{q}) = \left\{\begin{aligned}
& |\mathcal{C}_{\nu | \tilde{q}}^{-1/2}(y - \mathcal{G}_{a}(\tilde{q}) - \bar{\nu}_{\tilde{q}})|^{2} + \|\tilde{q}\|_{E}^{2}
\quad \text{if }\tilde{q}\in E, \text{ and}\\
& +\infty \quad\quad\quad\quad\quad\quad\quad\quad\quad\quad\quad\quad\,\,\,\, \text{else,}
\end{aligned}
\right.
\end{align}
with $E = A^{-(s+1)}H^{1+s}(D).$ According to Theorem \ref{MAP1theorem} and Theorem \ref{MAP2theorem},
we can calculate the minimizers of the function $I(\tilde{q})$ to obtain some appropriate estimators.
With this observation, it seems that we could design algorithms by employing ideas used in \cite{Bao3,Bayesian8}.
However, the present work focuses on the theoretical foundations.
For designing practical algorithms, we will report it in our future work.

\section{Conclusion}

In this article, we provide a fractional Helmholtz equation, then formulate two scattering problems: one is related
to the loss-dominated fractional Helmholtz equation, another one is related to the dispersion-dominated fractional Helmholtz equation.
For the former one, a well-posedness theory has been established for general wavenumber $k > 0$ and the Lipschitz continuity
of the solution with respect to the scatterer has also been proved.
For the later one, because the problem seems too complex, we only provide a well-posedness theory for sufficiently small wavenumber.
For the general wavenumber, the problem needs further investigations and it may be related to the studies on
regularity properties about fractional elliptic systems.

In order to study an inverse scattering problem related to the loss-dominated fractional Helmholtz equation,
we generalize the traditional infinite-dimensional Bayesian method to
the infinite-dimensional Bayesian model error method, which allows a part of the noise to depend on the
target function (the function needs to be estimated). A result similar to the posterior consistency has been obtained,
and the relationship between the Bayesian methods and the regularization methods has also been discussed.
In the end, general theory has been applied to our inverse scattering problem.

There are numerous further problems, e.g., designing an algorithm for inverse problems with this new model;
generalizing our theory under the variable Besov prior proposed in a recent article \cite{Jia1}.

\section{Acknowledgements}

The authors would like to thank the anonymous
referees for their comments and suggestions, which helped to improve the paper significantly.
This work was supported in part by NSFC under Contact 41390454, Contact 11501439 and Contact 11131006.
in part by postdoctoral science foundation project of China under Contact 2015M580826,
in part by EU FP7 project LIVCODE (295151).
The first author would like to thank Dr. Chenchen Mou and Jiaqing Yang's helpful discussions.

\bibliographystyle{plain}
\bibliography{references}

\end{document}